\documentclass{article}
\usepackage{harvard}
\usepackage{graphicx}
\usepackage{camnum}
\usepackage{amsmath}
\usepackage{enumerate}
\usepackage{algorithm}
\usepackage{algpseudocode}
\usepackage{tikz,pgfplots}
\usepackage{tikz-cd}

\makeatletter
\newsavebox{\@brx}
\newcommand{\llangle}[1][]{\savebox{\@brx}{\(\m@th{#1\langle}\)}%
	\mathopen{\copy\@brx\kern-0.5\wd\@brx\usebox{\@brx}}}
\newcommand{\rrangle}[1][]{\savebox{\@brx}{\(\m@th{#1\rangle}\)}%
	\mathclose{\copy\@brx\kern-0.5\wd\@brx\usebox{\@brx}}}
\makeatother

\allowdisplaybreaks

\newcommand{\ee}{\mathrm{e}}
\newcommand{\ii}{{\mathrm i}}
\newcommand{\DDD}{\DD{D}}

\newtheorem{definition}[theorem]{Definition}

\title{T-systems: a theory of orthonormal functions with a tridiagonal differentiation matrix}
\author{Arieh Iserles\\
	Department of Applied Mathematics and Theoretical Physics\\
	University of Cambridge
	\and 
	Marcus Webb\\
	Department of Mathematics\\University of Manchester}

\begin{document}
\maketitle
\tableofcontents
\newpage
\begin{abstract}
  The starting point of this paper is that a spectral method is essentially a combination of an orthonormal basis of the underlying Hilbert space with Galerkin conditions. The choice of an orthonormal basis depends on a number of desirable features which we explore in the context of spectral methods for time-dependent partial differential equations in a single space dimension.
  
  A central role in ensuring many of the above features is played by the {\em differentiation matrix\/} of the underlying orthonormal system. In particular, it is beneficial if this matrix is skew-symmetric and tridiagonal. While  orthonormal systems with this feature have been characterised in A.~Iserles \&~M.~Webb, ``Orthogonal systems with a skew-symmetric differentiation matrix'', {\em Found.\ Comput.\ Maths\/} {\bf 19} (2019), 1191--1221, employing Fourier transforms, in this paper we provide an alternative characterisation using the {\em differential Lanczos algorithm,\/} which can be implemented constructively. It is valid for inner products that obey an `integration-by-parts condition', inclusive of $\CC{L}_2$ and Sobolev norms on the real line. 

Motivated by quest for integration methods that conserve Hamiltonian energy, we conclude the paper replacing inner products by more general sesquilinear forms and presenting preliminary results. Here the Fourier transform characterisation generalises to spectral theory of Schr\"odinger operators and the differential Lanczos algorithm generalises to the differential Arnoldi algorithm.
\end{abstract}

\section{Introduction}

\subsection{Spectral methods for time-dependent PDEs}

The original motivation for the present work is the design of numerical methods for time-dependent dispersive equations of the form
\begin{equation}
  \label{eq:1.1}
  \ii \frac{\partial u}{\partial t}=-\Delta u+G(\MM{x},t,u)u,\qquad t\geq0,\quad \MM{x}\in\BB{R}^d,
\end{equation}
where $u:\BB{R}^d\times[ 0,\infty)\rightarrow\BB{C}$ and we are given a sufficiently regular initial value $u_0$ at $t=0$. The given function $G : \BB{R}^d \times [0,\infty) \times \BB{C} \to \BB{R}$ is also assumed to be sufficiently regular. Important examples, all with significant applications in quantum mechanics, include
\begin{itemize}
\item $G(\MM{x},t,u)=V(\MM{x})$: the linear Schr\"odinger equation,
\item $G(\MM{x},t,u)=\lambda |u|^2$: the nonlinear Schr\"odinger equation, and
\item $G(\MM{x},t,u)=V(\MM{x},t)+\lambda|u|^2$: the Gross--Pitaevskii equation.
\end{itemize}
It is well known that the evolution of \R{eq:1.1} is unitary, i.e.~$\|u(\,\cdot\,,t)\|_{\CC{L}_2(\bb{R})}\equiv \|u_0\|_{\CC{L}_2(\bb{R})}$, a property of central importance in physical applications, because it corresponds to the fact that $|u(\,\cdot\,,t)|^2$ is a probability distribution. It is thus, in the spirit of geometric numerical integration \cite{blanes16cig,hairer06gni}, natural to respect it under discretisation. An added bonus is that a unitary discretisation is automatically stable in the $\CC{L}_2(\BB{R})$ norm and hence by the Lax Equivalence Theorem \cite[p.~360]{iserles09fcn}, if the numerical method is also consistent, then it is convergent. 

The equation \R{eq:1.1}, defined for $\MM{x}\in\BB{R}^d$, represents a natural setting for both analysis and computation of dispersive equations of quantum mechanics \cite{lasser20cqd}. It is often presented in different spatial domains, in particular on the torus $\BB{T}^d$, but this is often motivated by the quest for simpler analysis and faster numerical solution, both through the agency of Fourier series and the Fast Fourier Transform, rather than by genuine physical considerations. 

The `free space' setting, $\BB{R}^d$, presents considerable challenges for numerical discretisations. Naive use of finite differences or finite elements leads to a countable number of degrees of freedom. Restricting the domain to a sufficiently large box (with either zero Dirichlet or periodic boundary conditions), but how large is `large enough'?\footnote{Cf.\ a preliminary dicussion of this approach in \cite{iserles24cmw}.} This, in particular, is highly relevant to dispersive equations: in general we can say very little about the support of the solution as time evolves. 

An important exception to the principle of ``free space setting results in an infinite number of degrees of freedom'' is presented by {\em spectral methods.\/} While more familiar for pure boundary-value problems, spectral methods can be used also for time-dependent PDEs of the form
\begin{equation}
  \label{eq:1.2}
  \frac{\partial u}{\partial t}=\mathcal{L}[u]+f(\MM{x},t, u),\qquad (\MM{x},t)\in\Omega\times\BB{R}_+,
\end{equation}
where $\Omega\subseteq\BB{R}^d$, $f$ is a sufficiently regular function and $\mathcal{L}$ is a possibly nonlinear differential operator defined on a function space with appropriate boundary conditions on $\partial\Omega$ \cite{hesthaven07smt}. Let $\Phi=\{\varphi_n\}_{n\in\bb{Z}_+}$ be a sufficiently smooth orthonormal basis of $\CC{L}_2(\Omega)$,
\begin{displaymath}
  \langle \varphi_m,\varphi_n\rangle=\int_\Omega \varphi_m(\MM{x})\varphi_n(\MM{x})\D\MM{x}=\delta_{m,n},\qquad m,n\in\BB{Z}_+,
\end{displaymath}
and suppose that each $\varphi_n$ shares the same boundary conditions as \R{eq:1.2}. A spectral method approximates the solution as a truncated expansion in the basis $\Phi$,
\begin{equation}
  \label{eq:1.3}
  u(\MM{x},t)\approx u_N(\MM{x},t)=\sum_{n=0}^N \hat{u}_n(t) \varphi_n(\MM{x}).
\end{equation}

To determine the coefficients $\hat{u}_n$ we may use  Galerkin conditions which constrain the residual of the numerical solution to be orthogonal to the approximation space spanned by $\varphi_0,\ldots,\varphi_N$. By orthonormality of $\Phi$, this is equivalent to
\begin{equation}
  \label{eq:1.4}
  \frac{\D\hat{u}_m}{\D t}=\left\langle \mathcal{L}\sum_{n=0}^N \hat{u}_n\varphi_n+f,\varphi_m\right\rangle\!,\qquad m=0,\ldots,N.
\end{equation}
Note that, if $\mathcal{L}$ is linear, \R{eq:1.4} simplifies to
\begin{displaymath}
  \frac{\D\hat{u}_m}{\D t}=\sum_{n=0}^N \hat{u}_n \langle \mathcal{L}\varphi_n,\varphi_m\rangle+\langle f,\varphi_m\rangle=0,\qquad m=0,\ldots,N,
\end{displaymath}
which, for $f=f(\MM{x},t)$, is a linear ODE system. The initial conditions for \R{eq:1.4} can be obtained by expanding the original initial condition in the basis $\Phi$,
\begin{displaymath}
  \hat{u}_n=\int_\Omega u(\MM{x},0)\overline{\varphi_n(\MM{x})}\D\MM{x},\qquad n=0,\ldots,N.
\end{displaymath}
An important insight is that the boundedness (or otherwise) of $\Omega$ plays no role whatsoever in the definition of a spectral method. This makes spectral methods particularly suitable for problems of the type \R{eq:1.1}, in particular once the initial condition $u_0$ is very smooth -- ideally, analytic. It does not eliminate the difficulties entirely --- it merely delegates them to the computation of the associated integrals.

There are five properties on which a time-stepping method for PDEs is typically judged:
\begin{description}
\item[Stability:] The numerical evolution operator must be well posed. This is not an optional extra because, unless a method is stable, its convergence is not assured.
\item[Unitarity:] if the exact solution is unitary, i.e.\ conserves the $\CC{L}_2$ norm, ideally we want our method to inherit this feature. Note that unitarity implies stability.
\item[Conserved quantities:] Many time-dependent PDEs conserve quantities such as a Hamiltonian energy or momentum as $t$ evolves. Ideally, a numerical solution should either respect this feature or at least do so in some approximate sense in long time intervals \cite{hairer06gni}. Unitarity is a special case of this property.
\item[Fast convergence:] The cost of time-stepping  Galerkin equations \R{eq:1.4} typically scales as a power of $N$. Therefore, a critical advantage rests with orthogonal systems yielding approximations that converge rapidly, hence require small $N$. This is the well known trade-off of computational PDE theory: finite differences and finite elements result in very large algebraic (or ODE) systems which are sparse, while spectral methods often yield small, yet dense, systems. There can be exceptions to this (see \cite{olver2013fast} and \cite{olver20fau}).
\item[Rapid time stepping:] Equations \R{eq:1.4} need be solved by a time-stepping ODE solver, perhaps using a splitting method \cite{mclachlan02sm}. While much depends on the nature of the original PDE and on the ODE solver, a critical ingredient is fast computation of an expansion in the basis $\Phi$ (and the inverse operation of evaluating an expansion on a grid of points). This is the secret for the attraction of periodic boundary conditions, that allow for the use of a fast Fourier transform, to determine the first  $N$ coefficients in $\O{N\log_2N}$ operations.
\end{description}

Unless we are fortunate enough to have periodic boundary conditions, stability of a time-stepping spectral method is the first stumbling block. Once applied to boundary-value problems in domains with, say, Dirichlet boundary conditions, a common choice of a basis is orthogonal polynomials: spectral methods using orthogonal polynomials are most familiar on intervals and their tensor products, and recently there have been efforts to extend this to simplexes \cite{dunkl01ops,olver20fau}. However, at least in a na\"ive implementation, orthogonal polynomials do not generally yield stable and well-conditioned spectral methods, hence fail already at the first (and most important) hurdle. 

From now on we assume $d=1$, in other words that there is just one space variable. It is straightforward to generalise the material of this paper by tensorisation to $\BB{R}^d$, $\BB{T}^d$ and parallelepipeds for all $d\in\BB{N}$. It may also be possible to generalise to $d \geq2$ using sparse grids, but this is more technically involved and deferred to future work.

\subsection{Differentiation matrices}

Understanding why orthogonal polynomials can lead to unstable spectral methods for time-dependent PDEs will help us to develop and describe an alternative.

Let $\Phi=\{\varphi_n\}_{n\in\bb{Z}_+}\subset\CC{C}^1(a,b)$ be an orthonormal and complete basis of the Hilbert space $\mathcal{H}$, rich enough so that $\CC{L}_2(a,b)\subseteq\mathcal{H}$. Since $\varphi_m'\in C(a,b) \subset \mathcal{H}$, it follows that
\begin{displaymath}
  \varphi_m'=\sum_{n=0}^\infty \DDD_{n,m}\varphi_n,\qquad m\in\BB{Z}_+,
\end{displaymath}
for some coefficients $\DDD_{n,m} \in \BB{C}$. The infinite-dimensional {\em differentiation matrix\/} $\DDD$ is a fundamental concept in the application of orthonormal bases to spectral methods \cite{hesthaven07smt,trefethen2000sm} and, roughly speaking, it plays the role of the derivative operator in a setting whereby a function $f:\mathcal{H}\rightarrow\BB{C}$ is represented by its expansion coefficients in the basis $\Phi$: if $f=\sum_{n=0}^\infty  \hat{f}_n\varphi_n$ then $f'=\sum_{n=0}^\infty  \hat{f}_n\varphi_n'=\sum_{n=0}^\infty  (\DDD \hat{\MM{f}})_n\varphi_n$.

Subject to the right boundary conditions (see Subsection~2.1), the differential operator is skew-self-adjoint. Given that the differentiation matrix replaces the differential operator in the discretised world, we might also wish it to be skew-self-adjoint, i.e.\ $\DDD^*=-\DDD$. This is an immensely fruitful assumption, which we illustrate by a number of examples.

The simplest possible example is the {\em advection equation\/}
\begin{displaymath}
  \frac{\partial u}{\partial t}=\frac{\partial u}{\partial x},\qquad t\geq0,\quad x\in\BB{R},
\end{displaymath}
whose solution is a unilateral shift, $u(x,t)=u(x+t,0)$. Its spectral approximation is $\hat{\MM{u}}'=\DDD\,\hat{\MM{u}}$ (we might truncate the equations but this makes no difference to our argument), whose solution is $\hat{\MM{u}}(t)=\ee^{t\mbox{\smallcurly D}}\hat{\MM{u}}(0)$. Since $\DDD$ is skew-self-adjoint, $\ee^{t\mbox{\smallcurly D}}$ is unitary and (in the $\ell_2$ norm) $\|\hat{\MM{u}}(t)\| = \|\hat{\MM{u}}_0\|$. Hence we have stability and, even better, unitarity. 

Our next example is the diffusion equation
\begin{displaymath}
  \frac{\partial u}{\partial t}=\frac{\partial }{\partial x} \!\left[a(x) \frac{\partial u}{\partial x} \right]\!,\qquad t\geq0,\quad x\in\BB{R},
\end{displaymath}
where $a(x) > 0$ for all $x$. Our spectral method is $\hat{\MM{u}}'=\DDD \mathcal{A} \DDD\,\hat{\MM{u}}$, where $\mathcal{A}$ is a symmetric positive definite matrix (since $\Phi$ is an orthonormal basis) that depends on $a(x)$ and the choice of $\Phi$. $\DDD$ being skew-self-adjoint implies that $\DDD\mathcal{A}\DDD$ is symmetric and negative semi-definite, so that $\D\|\hat{\MM{u}}(t)\|/\D t\leq0$ and the $\ell_2$ norm of the solution is a decreasing function, just as the $\mathcal{L}_2$ norm of the exact solution. 

Our final example is the dispersive equation \R{eq:1.1}. After discretisation,
\begin{displaymath}
  \ii \hat{\MM{u}}'=-\DDD^2\, \hat{\MM{u}}+\mathcal{G}(t,\hat{\MM{u}})\hat{\MM{u}},
\end{displaymath}
where $\mathcal{G}$ is a self-adjoint matrix whose exact form depends on $G$ and $\Phi$. Therefore
\begin{equation}
  \label{eq:1.5}
  \frac{\D\|\hat{\MM{u}}\|^2}{\D t}={\hat{\MM{u}}}^* \hat{\MM{u}}'+{({\hat{\MM{u}}')}^{*}}\hat{\MM{u}}={\hat{\MM{u}}}^*(\ii\DDD^2-\ii{\DDD^{*}}^2)\hat{\MM{u}}- {\hat{\MM{u}}}^*[\ii\mathcal{G}(t,\hat{\MM{u}}) - \ii\mathcal{G}(t,\hat{\MM{u}})^*]\hat{\MM{u}}=0.
\end{equation}
We deduce that, like \R{eq:1.1}, the solution of \R{eq:1.5} is unitary. What is perhaps even more remarkable is that the proof of unitarity (and stability) is {\em exactly\/} the same as the proof of unitarity (and well posedness) of the original PDE, just replacing the differential operator by a differentiation matrix and  $\CC{L}_2$ by $\ell_2$ norm. 

Once our differentiation matrix is skew-symmetric, we automatically gain both stability and unitarity for a wide range of initial-value problems. Of course, the other three desiderata need not be automatically satisfied.  We address Hamiltonian invariants in Section~4. Insofar as the speed of convergence is concerned, relatively little is known: in general, the speed of convergence in $\CC{L}_2(\BB{R})$ by orthogonal sequences (as opposed to a compact interval) is a largely open issue. Partial results, of relevance to dispersive equations, are reported in \cite{iserles22awp} but in general there is no comprehensive theory. Rapid time stepping has a number of aspects. Firstly, it is convenient to work with tridiagonal differentiation matrices: in Sections~2--3 we show that orthogonal bases with such matrices can be characterised completely in three situations: Cauchy, periodic and zero Dirichlet boundary conditions, and this material forms the bulk of this paper. Secondly, it is possible to characterise {\em all\/} orthogonal bases so that $\DDD$ is skew symmetric, tridiagonal and the expansion can be computed by either fast Fourier or fast cosine or fast sine transform. There are exactly five such bases, which have been described in \cite{iserles21fco}.  Other strategies for fast time stepping include splitting methods \cite{bader14eas,mclachlan02sm,blanes2024splitting} and special methods tailored for high-dimensional Schr\"odinger-type equations like Hagedorn functions, Gaussian beams and the Gaussian wave-packet transform \cite{lasser20cqd}.

In this paper we describe a toolbox for stable spectral methods in a number of settings. Perhaps the most important is the real line, but our theory has been extended also to bounded intervals with periodic boundary conditions. Although in principle we can also use it in a compact interval or in $[a,\infty)$ for some $a\in\BB{R}$, our construction appears to deliver orthonormal bases with poor approximation properties. In that case one may consider an alternative approach, described in \cite{iserles23ost}, whereby the differentiation matrix $\DDD$ is no longer tridiagonal (or banded) but it has a different helpful structure, allowing to form the first $N$ terms of  $\DDD\,\hat{\MM{u}}$ in $\O{N}$ operations. 

In Section 2 we sketch the theory of orthonormal systems with tridiagonal, skew-Hermitian differentiation matrices using the Fourier transform. This theory has been already described in \cite{iserles19oss} for $\CC{L}_2(\BB{R})$ and in \cite{iserles23sos} for Sobolev spaces on the real line with Cauchy boundary conditions, here we extend the setting to other types of boundary conditions. In Section~3 we introduce a new, alternative approach towards the construction of such systems, the {\em differential Lanczos algorithm.\/} We accompany our construction with a number of detailed examples. Finally, in Section~4 we replace inner products by sesquilinear forms. This is motivated by our quest to explore preservation of Hamiltonian energy by spectral methods for the linear Schr\"odinger equation. We prove that, in general, it is impossible to conserve both unitarity and Hamiltonian energy in this setting. Moreover, preservation of Hamiltonian energy leads to upper-Hessenberg differentiation matrices which are neither tridiagonal nor skew symmetric. However, at least in the  examples we have explored, such a differentiation matrix is {\em almost\/} tridiagonal and skew symmetric, a phenomenon whose importance and detailed analysis are matter for future research.

\setcounter{equation}{0}
\setcounter{figure}{0}
\section{On differentiation matrices}

\subsection{Skew-Hermitian differentiation matrices}

The connection between skew-self-adjointness of an operator, and skew-Hermitian-ness of its matrix with respect to a basis for the underlying space depends intimately on the inner product (more generally a duality pairing, but we will not discuss such further abstractions). The setting for our theory is therefore chosen to be a Hilbert space $\mathcal{H}$ with inner product $\langle \cdot , \cdot \rangle$, upon which the differential operator can be defined at least densely on $D(\DDD) \subset \mathcal{H}$. The following property is crucial.
\begin{definition}
  We say that the Hilbert space $\cal{H}$ satisfies the {\em Integration-by-Parts (IbP) property\/} if 
  \begin{equation}
    \label{eq:2.1}
    \Re(\langle f',g\rangle+\langle f,g'\rangle)=0,\qquad f,g\in\mathcal{H}.
\end{equation}
\end{definition}

In the special case when $\cal H$ is a subspace of $\CC{L}_2(a,b)$ with the standard inner product (where $a = -\infty$ or $b=\infty$ are possible), integration by parts yields
\begin{Eqnarray*}
  \langle f',g\rangle &=&\int_a^b f'(x)\overline{g(x)}\D x\\
  &=& \left[f \, \overline{g} \right]_a^b -\int_a^b f(x)\overline{g'(x)}\D x \\
  &=& f(b)\overline{g(b)} - f(a)\overline{g(a)}  - \langle f,g'\rangle.
\end{Eqnarray*}
Therefore, with this inner product, the IbP property is equivalent to 
\begin{equation}
  \label{eq:2.2}
  \Re [f(b)\overline{g(b)} - f(a) \overline{g(a)}]=0,
\end{equation}
for all $f,g\in \cal H$.

The condition \R{eq:2.2} and its generalisation to a Sobolev setting are consistent with three  scenarios:
\begin{enumerate}
	\item $a=-\infty$, $b=\infty$ and $\mathcal{H} \subseteq \CC{L}_2(\BB{R})$ implies Cauchy boundary conditions. This is the setting of \cite{iserles19oss,iserles20for,iserles21fco,iserles23sos} which we will revisit briefly in Subsection~2.2;
	\item $(a,b)\subset\BB{R}$ with zero Dirichlet boundary conditions;
	\item Periodic boundary conditions in a bounded interval $(a,b)$. As a matter of fact, we may allow a mildly more general setting of {\em twisted periodicity,\/} whereby there exists $\kappa\in(-\pi,\pi]$ such that $f(a)=\ee^{\ii\kappa} f(b)$ for every $f\in{\cal H}$. This has applications to Bloch waves in solid state physics \cite{kittel2015introduction} and in the theory of sound waves \cite{maierhofer2020wsi,posson10ouu}.
\end{enumerate}

Yet, IbP is true in a considerably more general Hilbert spaces. Thus, let $v_k\geq0$, $k\in\BB{Z}_+$, not all zero. Following \cite{iserles23sos}, we let 
\begin{equation}
  \label{eq:2.3}
  \langle f,g\rangle_v =\sum_{k=0}^\infty v_k  \int_{a}^b f^{(k)}(x) \overline{g^{(k)}(x)}\D x.
\end{equation}
whereby the Hilbert space $\mathcal{H}=\mathcal{H}_v$ comprises  all functions $f$ such that $\langle f,f\rangle_v<\infty$. 

It is helpful to define $v(\xi)=\sum_{k=0}^\infty v_k \xi^{2k}$. The proof of the IbP property is a straightforward generalisation of the special case $v(\xi)\equiv1$, which corresponds to the $\CC{L}_2(a,b)$ inner product. In particular,  $v(\xi)=\sum_{k=0}^p \xi^{2k}$ results in orthogonality in the {\em Sobolev space\/} $\CC{H}^p(a,b)$. 

\subsection{T-systems}\label{sec:Tsystems}

Let $\tilde{\Phi} = \{\tilde{\varphi}_n\}_{n = 0}^\infty$ be an orthonormal basis of $\CC{L}_2(a,b)$ (where $a = -\infty$ or $b=\infty$ are possible) and assume that $\tilde{\varphi}_n\in\CC{C}^\infty(a,b)$. We say that it is a {\em T-system\/} if the differentiation matrix $\tilde{\DDD}$ is both skew-Hermitian and tridiagonal\footnote{T stands for tridiagonal.}. In other words,
\begin{equation}
  \label{eq:2.4}
  \tilde{\varphi}_n'=-\overline{b}_{n-1}\tilde{\varphi}_{n-1}+\ii c_n \tilde{\varphi}_n+b_n \tilde{\varphi}_{n+1},\qquad n\in\BB{Z}_+,
\end{equation}
where $c_n\in\BB{R}$ and $b_{-1}=0$. We assume that $b_n\neq0$, $n\in\BB{Z}_+$, so that $\widetilde{\DDD}$ is irreducible. 

The main reason why tridiagonality is important is that it leads to affordable linear algebra. Note that the matrix $\widetilde{\DDD}$, acting on expansion coefficients of $f\in\mathcal{H}$, plays the same role as the differential operator applied to $f$, hence we `replace' differentiation by a multiplication with (in practice, truncated) tridiagonal matrix.

Suppose that $b_n=|b_n|\ee^{\ii\theta_n}$, $n\in\BB{Z}_+$, and set
\begin{displaymath}
  \varphi_n(x)=\exp\!\left(\ii\sum_{\ell=0}^{n-1}\theta_l\right)\!\tilde{\varphi}_n(x),\qquad n\in\BB{Z}_+.
\end{displaymath}
Substituting into \R{eq:2.4}, simple algebra confirms that 
\begin{equation}
  \label{eq:2.5}
  \varphi_n'=-|b_{n-1}|\varphi_{n-1}+\ii c_n \varphi_n+|b_n|\varphi_{n+1},\qquad n\in\BB{Z}_+.
\end{equation}
Therefore, we can assume without loss of generality that $b_n>0$, $n\in\BB{Z}_+$ (hence replacing $|b_m|$ with $b_m$ in the above expression, as we do in the sequel). In other words, the main diagonal of $\DDD$ is pure imaginary, while its principal sub- and super-diagonals are strictly negative and strictly positive, respectively. 

We completely characterise in the sequel all the T-systems corresponding to Cauchy and periodic boundary conditions, while arguing that in the presence of zero Dirichlet boundary conditions (whether in a finite or semi-infinite interval) the approach of T-systems is unlikely to work. In the latter case we may adopt an alternative approach, introduced in \cite{iserles23ost}.

Commencing from \R{eq:2.5}, it is possible to represent each $\varphi_n$ in terms of $\varphi_0$ and its derivatives. This has been already pointed out (in a somewhat simplified setting) in \cite{iserles19oss} but without any tangible consequences: in this paper we demonstrate the deep importance of this mechanism. Thus, let $\{b_n\}_{n\in\bb{Z}_+}\in\BB{R}_+^\infty$ and $\{c_n\}_{n\in\bb{Z}_+}\in\BB{R}^\infty$ be given. Letting $n=0$ in \R{eq:2.5}, $\varphi_0'=\ii c_0 \varphi_0+b_0\varphi_1$, we solve for $\varphi_1$ and obtain $\varphi_1=b_0^{-1}(-\ii c_0\varphi_0+\varphi_0')$.  Next we set $n=1$, solve for $\varphi_2$, $\varphi_2=(b_0b_1)^{-1}[(b_0^2-c_0c_1)\varphi_0-\ii(c_0+c_1)\varphi_0+\varphi_0'']$ and continue in this vain for increasing $n$. It follows by trivial induction that
\begin{equation}
  \label{eq:2.6}
  \varphi_n(x)=\frac{1}{b_0b_1\cdots b_{n-1}}\sum_{\ell=0}^n \alpha_{n,\ell} \varphi_0^{(\ell)}(x),\qquad n\in\BB{Z}_+,
\end{equation}
where
\begin{Eqnarray*}
  \nonumber
  \alpha_{n+1,\ell}&=&b_{n-1}^2 \alpha_{n-1,0}-\ii c_n \alpha_{n,0} ,\\
  \alpha_{n+1,\ell}&=&\alpha_{n,\ell-1}+b_{n-1}^2\alpha_{n-1,\ell}-\ii c_n\alpha_{n,\ell},\qquad \ell=1,\ldots,n-1,\\
  \alpha_{n+1,n}&=&\alpha_{n,n-1}-\ii c_n\alpha_{n,n},\\
  \alpha_{n+1,n+1}&=&\alpha_{n,n}=1.
\end{Eqnarray*}
This approach has not been pursued further in \cite{iserles19oss} because, while it generates a set $\Phi$, in general it is not orthonormal. Moreover, in the case of Cauchy boundary condition (and, later in this paper, also in the case of periodic boundary conditions) there exists an alternative approach, which forms the theme of Subsections 2.3--4.  However, it turns out that \R{eq:2.6}, with proper setting and interpretation, the representation \R{eq:2.6} of fundamental importance and forms the spine of Section~3.

\begin{definition}
  Let $g\in\CC{C}^\infty(a,b)$. The linear space
  \begin{displaymath}
    {\DD K}_n(g)=\CC{Span}\,\{g,g',\ldots,g^{(n-1)}\}
  \end{displaymath}
  is called the $n$th {\em differential Krylov space.\/}
\end{definition}

The name is modelled upon the familiar Krylov subspaces  $\CC{Span}\,\{\MM{v},A\MM{v},\ldots,A^{n-1}\MM{v}\}$, where $A$ and $\MM{v}$ are a matrix and a vector, respectively, and which form the mainstay of modern numerical algebra \cite{golub96mc}. It is easy to verify that, unless $g(x)$ is a linear combination of functions of the form $p(x)\ee^{\alpha x}$ for polynomials $p$ and $\alpha\in\BB{C}$ (in other words, $g$ is a solution of an homogeneous linear ordinary differential equation with constant coefficients), it is always true that $\dim\DD{K}_n(g)=n$ \cite[Thm 2.4]{olver2009gmres}. In the context of this paper we deduce from \R{eq:2.6} that a T-function $\varphi_n$ lives in $\DD{K}_{n+1}(\varphi_0)$ for all $n\in\BB{Z}_+$.

\begin{definition}\label{def:seed}
	A smooth function $\varphi_0$ is a seed function if it is not the solution of an homogeneous linear ordinary differential equation with constant coefficients.
\end{definition}

\subsection{The Cauchy problem {\em via\/} the Fourier transform}

The standard approach to T-systems on the real line, pioneered in \cite{iserles19oss} for $\mathcal{H}=\CC{L}_2(\BB{R})$ and extended to Sobolev-like spaces ${\mathcal H}_v$ in \cite{iserles23sos}, uses the Fourier transform to turn \R{eq:2.5} into
\begin{equation}
  \label{eq:2.7}
  b_n \vartheta_{n+1}(\xi)=(\xi-c_n)\vartheta_n(\xi)-b_{n-1}\vartheta_{n-1}(\xi),\qquad n\in\BB{Z}_+,\quad \xi\in\BB{R},
\end{equation}
where $\vartheta_n(\xi)=(-\ii)^n\hat{\varphi}_n(\xi)$.

Experts in orthogonal polynomials would immediately identify \R{eq:2.7} with the three-term recurrence relation for orthonormal polynomials and recall the Favard theorem \cite[p.~21]{chihara78iop}, also known as the spectral theorem for orthogonal polynomials \cite[p.~30]{ismail05cqo}: Any system of polynomials that obeys \R{eq:2.7} and is such that $\vartheta_0$ is a constant and $\vartheta_{-1}$ is taken as zero, is an orthonormal polynomial sequence with respect to some Borel measure $\D\mu$. 

This process can be reversed: Let $\{p_n\}_{n\in\bb{Z}_+}$ be the set of orthonormal polynomials with respect to $\D\mu$ and such that the coefficient of $\xi^n$ in $p_n(\xi)$ is positive. We let
\begin{displaymath}
   \vartheta_n(\xi)=\sqrt{\frac{w(\xi)}{v(\xi)}} p_n(\xi),\qquad \xi\in\BB{R},
\end{displaymath}
where $w$ is the Radon--Nikodym derivative of $\mu$ (thus, $w(\xi)\D\xi=\D\mu(\xi)$ and map $\ii^n\vartheta_n$ to $\varphi_n$ with the inverse Fourier transform. The functions $\varphi_n$ obey \R{eq:2.5} (because of \R{eq:2.7}) and are orthonormal (as can be easily proved using the Plencharel theorem). However, the set $\Phi$ need not be a complete basis for $\mathcal{H}_v$: for example, if $v\equiv1$ and $\mathcal{H}_v=\CC{L}_2(\BB{R})$, the closure of $\Phi$ is a Payley--Wiener space of $\CC{L}_2(\BB{R})$ functions which are band-limited to the support of $\D\mu$. In other words, $\Phi$ is a basis of $\CC{L}_2(\BB{R})$ if and only if the support of $\D\mu$ is the entire real line. This generalises to Sobolev-like Payley--Wiener spaces \cite{iserles23sos}.  

\subsection{Periodic boundary conditions {\em via\/} the Fourier transform}

\subsubsection{The theory}

The obvious orthonormal system in the presence of periodic boundary conditions is the standard scaled Fourier system. The purpose of this subsection is not to offer an alternative to Fourier expansions, with their well-known advantages in computing fast expansion coefficients with FFT, just to add to the theory of T-systems and make it more comprehensive.

Without loss of generality we assume that the interval of periodicity is $[-\pi,\pi]$ and seek an orthonormal system $\Phi$, dense in $\mathcal{H}=\CC{L}_2[-\pi,\pi]\cap \CC{C}^{\CC{per}}[-\pi,\pi]$\footnote{Unless a function is continuous at the endpoints, it is impossible to talk of periodicity.} (for brevity we restrict the discussion to $\CC{L}_2$ setting, while mentioning that a generalisation to $\mathcal{H}_v$ is simple) and consistent with the differential recurrence \R{eq:2.5}. Each $\varphi_n$ can be expanded into Fourier series,
\begin{equation}
  \label{eq:2.8}
  \varphi_n(x)=\sum_{m=-\infty}^\infty r_{n,m}\ee^{\ii mx},\qquad n\in\BB{Z}_+,
\end{equation}
where
\begin{displaymath}
  r_{n,m}=\frac{1}{2\pi}\int_{-\pi}^\pi \varphi_n(x)\ee^{-\ii mx}\D x,\qquad m\in\BB{Z},\;\;n\in\BB{Z}_+.
\end{displaymath}
Substituting this in \R{eq:2.5} (so as to have a T-system), we easily deduce that
\begin{equation}
  \label{eq:2.9}
  \ii mr_{n,m}=-b_{n-1}r_{n-1,m}+\ii c_n r_{n,m}+b_n r_{n+1,m},\qquad m\in\BB{Z},\;\;n\in\BB{Z}_+.
\end{equation}

\begin{proposition}
  It is true that
  \begin{equation}
    \label{eq:2.10}
    r_{n,m}=\frac{\ii^n}{b_0b_1\cdots b_{n-1}} q_n(m) r_{0,m},\qquad m\in\BB{Z},\;\;n\in\BB{Z}_+,
  \end{equation}
  where each $q_n$ is an $n$th-degree monic polynomial that obeys the three-term recurrence relation
  \begin{equation}
    \label{eq:2.11}
    q_{n+1}(m)=(m-a_n)q_n(m)-b_{n-1}^2 q_{n-1}(m).
  \end{equation}
\end{proposition}

\begin{proof}
  Direct substitution in \R{eq:2.9}.
\end{proof}

It is helpful to convert a monic orthogonal system into an orthonormal one: according to \cite[p.~215--6]{chihara78iop}, the orthonormal polynomials consistent with the three-term recurrence relation \R{eq:2.9} are
\begin{displaymath}
  p_n(x)=\frac{q_n(x)}{b_0b_1\cdots b_{n-1}},\qquad n\in\BB{Z}_+.
\end{displaymath}
Therefore, we can rewrite \R{eq:2.10} in the form
\begin{equation}
  \label{eq:2.12}
  r_{n,m}=\ii^n p_n(m)r_{0,m},\qquad n\in\BB{Z},\;\; n\in\BB{Z}_+.
\end{equation}

Letting \R{eq:2.10} in \R{eq:2.8} does not guarantee orthogonality. However, by the Plencharel theorem (again!) it is true for every $f,g\in \CC{L}_2[-\pi,\pi]$ that
\begin{displaymath}
  \langle f,g\rangle =\frac{1}{2\pi} \int_{-\pi}^\pi \left(\sum_{k=-\infty}^\infty \hat{f}_k\ee^{\ii kx}\right)\!\overline{\left(\sum_{\ell=-\infty}^\infty \hat{g}_\ell \ee^{\ii\ell x}\right)}\D x=\sum_{k=-\infty}^\infty \hat{f}_k \overline{\hat{g}}_k.
\end{displaymath}
Therefore
\begin{displaymath}
  \langle\varphi_m,\varphi_n\rangle=\sum_{k=-\infty}^\infty r_{m,k}\overline{r}_{n,k}
\end{displaymath}
and, by virtue of \R{eq:2.10}, orthonormality of $\Phi$ is equivalent to 
\begin{equation}
  \label{eq:2.13}
  (-1)^n \ii^{m+n}\sum_{k=-\infty}^\infty |r_{0,k}|^2 p_m(k) p_n(k)=\delta_{m,n},\qquad m\in\BB{Z},\; n\in\BB{Z}_+.
\end{equation}

\begin{theorem}
  Let $\D\mu$ be an atomic measure on $\BB{R}$ with jumps of $w_n\geq0$ at $n\in\BB{Z}$, where $0<\sum_{n=0}^\infty w_n<\infty$, and set
  \begin{displaymath}
    w(x)=\sum_{n=0}^\infty w_n \delta(x-n),
  \end{displaymath}
  where $\delta$ is the delta function. Then
  \begin{displaymath}
      r_{0,m}=\sqrt{w_m},\qquad m\in\BB{Z}
  \end{displaymath}
  and the sequence
  \begin{equation}
    \label{eq:2.14}
    \varphi_n(x)= \ii^n \sum_{k=-\infty}^\infty \sqrt{w_k} p_n(k) \ee^{\ii kx},\qquad n\in\BB{Z}_+,
  \end{equation}
  is orthonormal.
\end{theorem}

\begin{proof}
  Using \R{eq:2.10}, we have
  \begin{displaymath}
    \langle \varphi_m,\varphi_n\rangle=\ii^{n-m} \sum_{m=-\infty}^\infty w_k p_m(k) p_n(k)=\ii^{n-m} \int_{-\infty}^\infty p_m(x)p_n(x)\D\mu(x)
  \end{displaymath}
   and the theorem follows. 
\end{proof}

The factor $\ii^n$ in front of $\varphi_n$ might appear somewhat worrisome, but this is illusory. Unless $c_n\equiv0$, the $\varphi_n$ {\em should\/} be complex-valued. On the other hand, once $c_n\equiv0$,  each $p_n$ retains the parity of $n$ and $w_{-k}=w_k$, $k\in\BB{Z}_+$. Therefore
\begin{Eqnarray}
  \label{eq:2.15}
  \varphi_{2n}(x)&=&(-1)^n \left[\sqrt{w_0}p_{2n}(0)+2\sum_{k=1}^\infty \sqrt{w_k} p_{2n}(k) \cos k x\right]\!,\\
  \nonumber
  \varphi_{2n+1}(x)&=&2(-1)^{n+1}\sum_{k=1}^\infty \sqrt{w_k} p_{2n+1}(k)\sin kx
\end{Eqnarray}
and $\Phi$ is real.

\subsubsection{Example: Simple periodic setting}

We did not provide any examples of T-systems for the Cauchy problem in Subsection~2.3 because they have been covered in depth in \cite{iserles19oss,iserles20for,iserles21fco,Iserles21daf}. The material of the current subsection is new and a couple of examples might be instructive.

Let $a\in(0,1)$. Our first example the atomic measure with jumps of $a^{|k|}$ at $k\in\BB{Z}$. We have evaluated the first few $p_n$s and $\varphi_n$s by brute force:
\begin{Eqnarray*}
  p_0(x)&\equiv&\sqrt{\frac{1-a}{1+a}},\\
  p_1(x)&=&\sqrt{\frac{(1-a)^3}{2a(1+a)}} x,\\
  p_2(x)&=&\sqrt{\frac{1-a}{2a(1+a)(1+8a+a^2)}} [-2a+(1-a)^2 x^2],\\
  p_3(x)&=&\sqrt{\frac{1-a}{12a}} \sqrt{\frac{1}{(1+a)(1+4a+a^2)}} [-(1+10a+a^2)x+(1-a)^2x^3].
\end{Eqnarray*}
For simplicity, we have computed the first few T-functions for $a=\frac12$:
\begin{Eqnarray*}
  \varphi_0(x)&=&\frac{\sqrt{6}}{3} \frac{1}{3\sqrt{2}-4\cos x},\\
  \varphi_1(x)&=&\frac{4\sqrt{3}}{3} \frac{\sin x}{(3\sqrt{2}-4\cos x)^2},\\
  \varphi_2(x)&=&\frac{2\sqrt{14}}{21} \frac{-26+27\sqrt{2}\cos x-12\cos^2x}{(3\sqrt{2}-4\cos x)^3},\\
  \varphi_3(x)&=&-\frac{8\sqrt{39}}{39} \frac{(11\sqrt{2}-16\cos x)(\sqrt{2}-\cos x)\sin x}{(3\sqrt{2}-4\cos x)^4}
\end{Eqnarray*}
-- a pattern of sorts emerges but we do not pursue it further in this paper.

\begin{figure}[htb]
  \begin{center}
    \includegraphics[width=160pt]{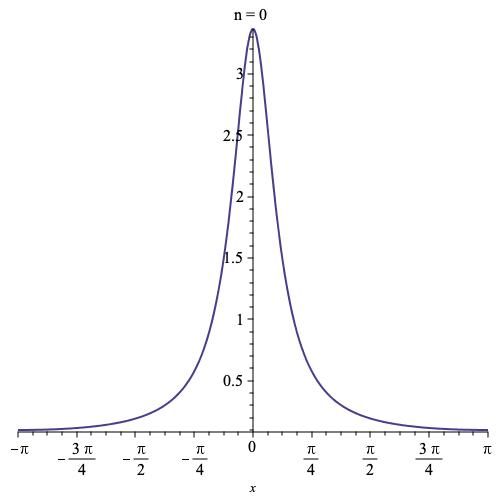}\hspace*{15pt}\includegraphics[width=160pt]{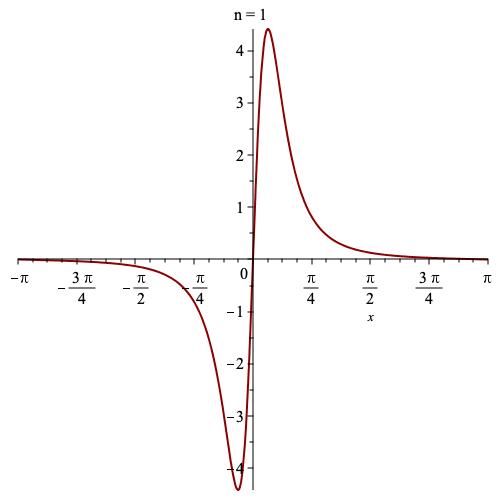}
    \includegraphics[width=160pt]{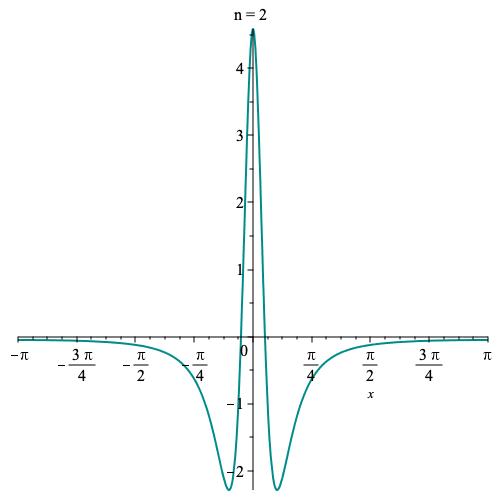}\hspace*{15pt}\includegraphics[width=160pt]{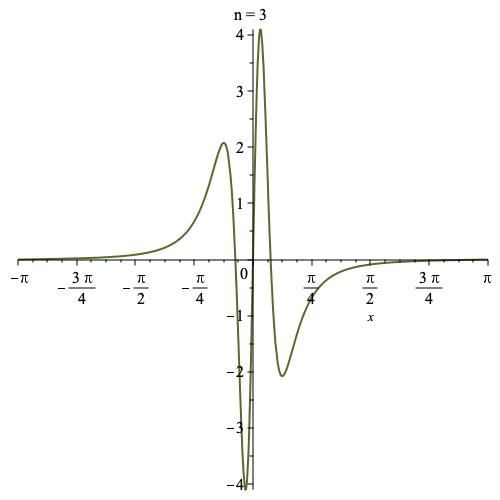}
    \caption{The functions $\varphi_n$, $n=0,1,2,3$, for the first periodic example and $a=\frac12$.}
    \label{fig:2.1}
  \end{center}
\end{figure}

The first few functions $\varphi_n$ are displayed in Fig.~\ref{fig:2.1}.

\subsubsection{Example: The Charlier measure}

Let $a>0$. The {\em Charlier measure\/} has jumps $w_k=\ee^{-a}a^k/k!$ at $k\in\BB{Z}_+$ \cite[p.~170--172]{chihara78iop}. The underlying monic {\em Charlier polynomials\/} have the form
\begin{equation}
  \label{eq:2.16}
  \CC{C}_n^a(x)=\sum_{k=0}^n {n\choose k}{x\choose k} k! (-a)^{n-k},
\end{equation}
exhibit the symmetry $\CC{C}_n^a(k)=\CC{C}_k^a(n)$, $k,n\in\BB{Z}_+$ \cite[18.21.2]{dlmf} and some important applications to Poisson point processes in probability theory.

We have emphasised in Subsection~2.2 that in the case of Cauchy  boundary conditions the underlying measure should be supported on the whole real line. In the current setting of periodic T-functions, the atomic measure $\D\mu$ must have jumps at all integer points. This is not the case for the Charlier measure but the situation is recoverable, along similar lines to the observation in \cite{iserles20for} that has led there from the Laguerre weight (supported in $[0,\infty)$) to the Malmquist--Takenaka T-system (dense in $\CC{L}_2(\BB{R})$). Thus, we append to $\{\varphi_n\}_{n\in\bb{Z}_+}$ its `mirror system', generated by the weight with jumps $w_k=\ee^{-a}a^{-k}/(-k)!$ at $k=0,1,\ldots$. The union of both systems, which we index in $\BB{Z}$, is complete in $\CC{L}_2(\BB{R})$. Its differentiation matrix $\DDD$ is tridiagonal but no longer irreducible, since $\DDD_{-1,0}=\DDD_{0,-1}=0$.

Since in the underlying measure $\|\CC{C}_n^a\|=\sqrt{\ee^{-a}a^nn!}$, we convert \R{eq:2.16} to an orthonormal system,
\begin{displaymath}
  p_n(x)=\sqrt{\frac{\ee^a}{n!a^n}}\CC{C}_n^a(x),\qquad n\in\BB{Z}_+.
\end{displaymath}
The three-term recurrence relation is
\begin{displaymath}
  \sqrt{a(n+1)} p_{n+1}(x)=(x-n-a) p_n(x)-\sqrt{an}p_{n-1}(x),\qquad n\in\BB{Z}_+.
\end{displaymath}

Using \R{eq:2.15}, we have
\begin{Eqnarray*}
  \varphi_{2n}(x)&=&\frac{\ee^{-a}}{a^n (2n)!}(-1)^n \left[\CC{C}_{2n}^a(0)+2\sum_{k=1}^\infty \frac{a^{k/2}}{\sqrt{k!}} \CC{C}_{2n}^a(k) \cos kx\right]\!,\\
  \varphi_{2n+1}(x)&=&\frac{2\ee^{-a}(-1)^{n+1}}{a^{n+\frac12}\sqrt{(2n+1)!}} \sum_{k=1}^\infty \frac{a^{k/2}}{\sqrt{k!}} \CC{C}_{2n+1}^a(k)\sin kx,\qquad n\in\BB{Z}_+,
\end{Eqnarray*}
with an obvious extension to $n\leq-1$. The above sums are not available in an explicit form.

\begin{figure}[htb]
  \begin{center}
    \includegraphics[width=160pt]{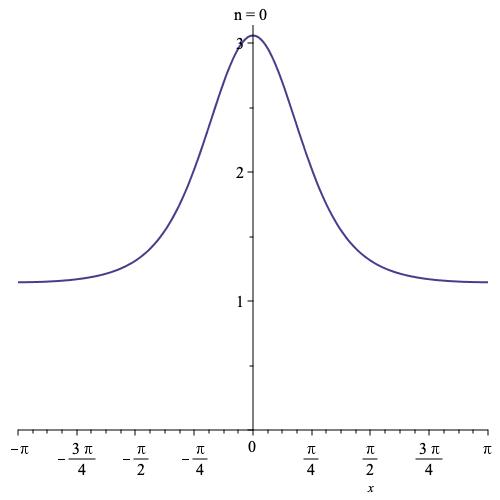}\hspace*{15pt}\includegraphics[width=160pt]{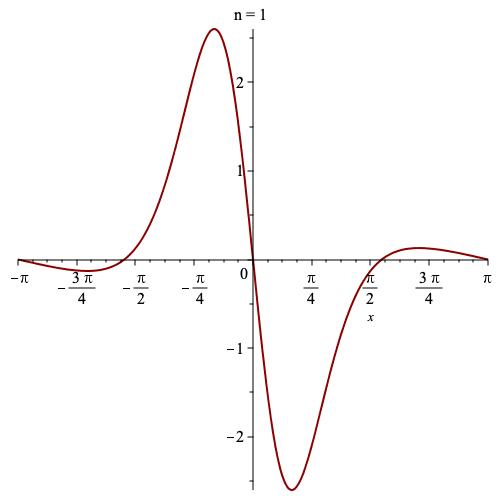}
    \includegraphics[width=160pt]{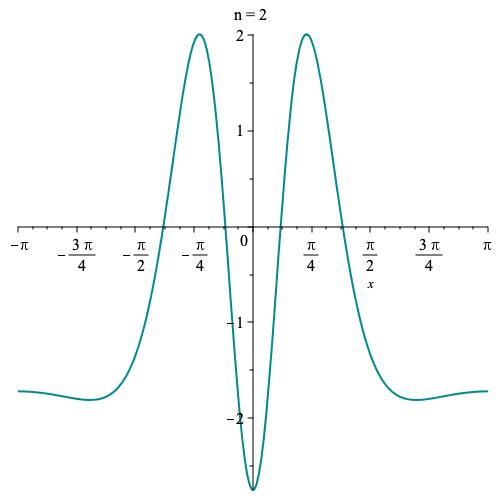}\hspace*{15pt}\includegraphics[width=160pt]{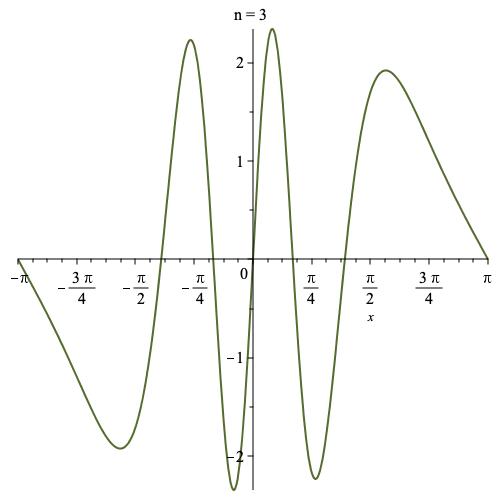}
    \caption{The functions $\varphi_n$, $n=0,1,2,3$, for the Charlier measure and $a=\frac12$.}
    \label{fig:2.2}
  \end{center}
\end{figure}

The first four Charlier T-functions are displayed in Fig.~\ref{fig:2.2}. 

There is a number of other measures which might be interesting in the current context. Restricting our attention to the {\em Askey scheme\/} \cite[18.21.1]{dlmf}, we might consider the Meixner measure. Three further measures in Askey's scheme -- Racah, Krawtchouk and Hahn -- have a finite number of points of increase at $n=1,2,\ldots,N$ for some $N$. They (together with their `mirror measures') generate finite sets of $T$-functions which, for sufficiently large $N$, might be suitable for approximation. We do not pursue this further in this paper. 

\subsubsection{Twisted periodic boundary conditions}

A generalisation from periodic to twisted-periodic boundary conditions in $[-\pi,\pi]$ is straightforward, all we need is to multiply the functions $\varphi_n$ in \R{eq:2.15} by $\ee^{\ii\kappa x/(2\pi)}$. 

\subsection{Zero Dirichlet boundary conditions}

Let $\mathcal{H}=\CC{C}^\infty[-1,1]$ and assume that $\varphi_n(\pm1)=0$, $n\in\BB{Z}_+$. In addition, we assume that $\varphi_0$ is analytic and note that, by virtue of \R{eq:2.6}, the same is true for all $\varphi_n$s. Recalling that $\MM{\varphi}'=\DDD\, \MM{\varphi}$, we deduce by induction that $\MM{\varphi}^{(r)}=\DDD^{\,r}\MM{\varphi}$. Moreover, $\DDD$ being tridiagonal, all its powers are bounded and we deduce that $\varphi_n^{(r)}(\pm1)=0$, $n,r\in\BB{Z}_+$. Analyticity implies that $\varphi_n$ is identically zero for all $n\in\BB{Z}_+$, which contradicts orthonormality. 

There are two remedies. One, which we do not pursue in this paper, is to give up on analyticity, at least at the endpoints. Specifically, we may impose an essential singularity at the endpoints, e.g.\ set
\begin{equation}
  \label{eq:2.17}
  \varphi_0(x)=a_0 \exp\!\left(-\frac{1}{1-x^2}\right)\!,\qquad x\in(-1,1),
\end{equation}
where $a_0\approx 2.741155146$ guarantees that $\|\varphi_0\|_{\CC{L}_2(-1,1)}=1$. 

\begin{figure}[htb]
  \begin{center}
    \includegraphics[width=120pt]{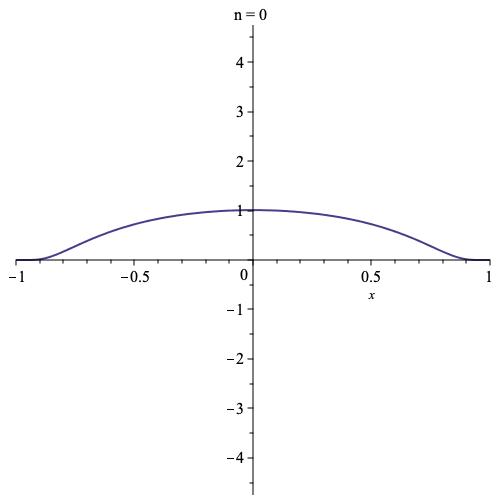}\hspace*{10pt}\includegraphics[width=120pt]{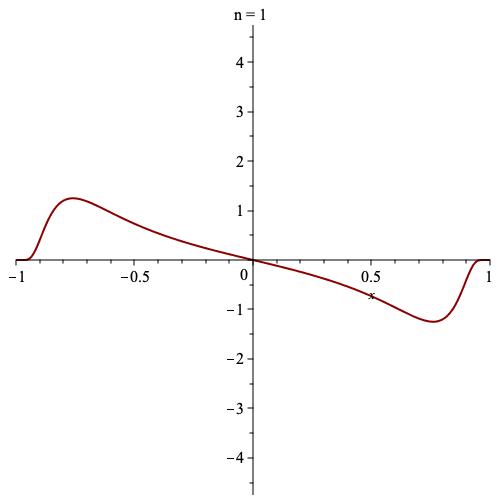}\hspace*{10pt}\includegraphics[width=120pt]{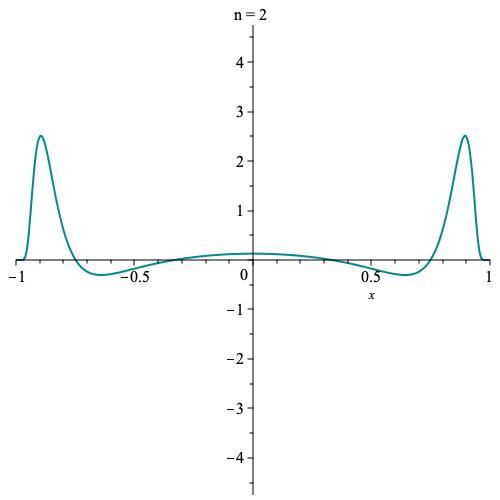}
    \includegraphics[width=120pt]{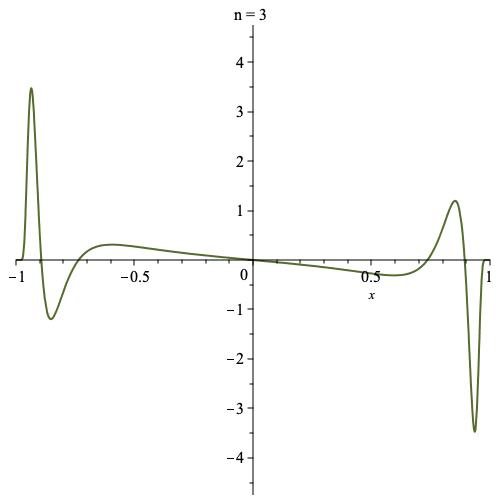}\hspace*{10pt}\includegraphics[width=120pt]{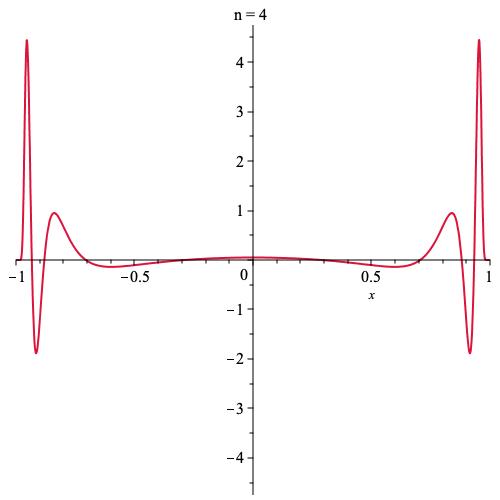}\hspace*{10pt}\includegraphics[width=120pt]{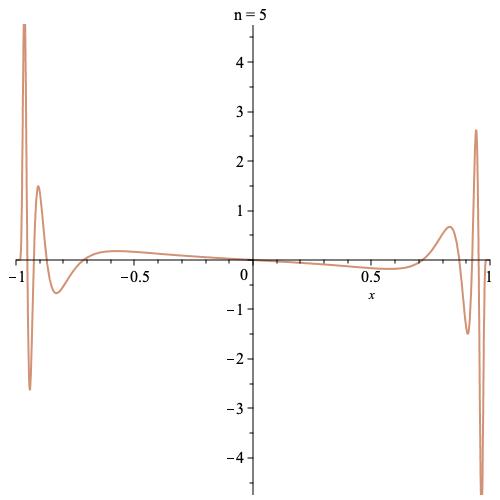}
    \caption{The functions $\varphi_n$, $n=0,\ldots,5$, for the $\varphi_0$ given by \R{eq:2.17}.}
    \label{fig:2.3}
  \end{center}
\end{figure}

Note that the $b_n$s are unknown (the $c_n$s vanish because of the symmetry of $\varphi_0$). Yet, in Section~3 we prove that it is enough to specify $\varphi_0$ to generate both $\Phi$ and $\{b_n\}_{n\in\bb{Z}+}$ with the {\em differential Lanczos algorithm.\/} Deferring a more detailed example to Subsection~\ref{sec:DLexamples}, we just comment upon Fig.~\ref{fig:2.3}, where the first six $\varphi_n$s are displayed, all to the same scale. It is evident that, while all the functions vanish (as required) at $\pm1$ and are smooth, they all share an essential singularity at the endpoints. The `compromise' between an essential singularity at $\pm1$, regularity inside and orthogonality appears to be a boundary layer, of an amplitude increasing with $n$, near the endpoints. For example, $\varphi_{10}(0.98356)\approx 8.9786443$, while $\max_{|x|\leq 0.7} |\varphi_{10}(x)|\approx 0.02886292$, a ratio exceeding 300. Functions like this have no `approximation power' and in all likelihood are useless as a practical tool. 

In general, T-functions appear to be of limited merit in the presence of zero Dirichlet boundary conditions. The reason is, paradoxically, that a tridiagonal $\DDD$ is too `nice': all its powers are bounded. A different approach to this problem is presented in \cite{iserles23ost}.
 
\setcounter{equation}{0}
\setcounter{figure}{0}
\section{The differential Lanczos algorithm}

\subsection{From $\varphi_0$ to $\Phi$}

We assume an inner product $\langle\,\cdot\,,\,\cdot\,\rangle$, acting on ${\cal H}\times{\cal H}$ and obeying the Integration-by-Parts property (IbP), $\Re(\langle f',g\rangle+\langle f,g'\rangle)=0$ for all $f,g\in\mathcal{H}$, and assume that  $\varphi_0$ is a seed function (as in Definition \ref{def:seed}) that lies on the unit sphere of $\mathcal H$. At the outset for simplicity's sake we restrict the narrative to $\mathcal{H}=\CC{L}_2(\BB{R})$, but this will be generalised later.

In Subsection~\ref{sec:Tsystems} we described a mechanism that, given $\varphi_0$, $\MM{b}=\{b_n\}_{n\in\bb{Z}_+}$ and $\MM{c}=\{c_n\}_{n\in\bb{Z}_+}$, generates the sequence $\varphi_1,\varphi_2,\ldots$ using the formula \R{eq:2.6}. Each $\varphi_n$ resides in the differential Krylov subspace ${\DD K}_{n+1}(\varphi_0)$, however, in general, the sequence $\Phi$ generated by \R{eq:2.6} fails to be orthonormal. The reason is that specifying {\em both\/} $\varphi_0$ {\em and\/} $\MM{b},\MM{c}$ overdetermines the system insofar as orthonormality is concerned. We have seen one aspect of this in Subsection~2.3: once $\MM{b},\MM{c}$ are given -- equivalently, in the format of the {\em Jacobi matrix\/},
\begin{displaymath}
  \mathcal{J}=\left[
  \begin{array}{ccccc}
         c_0 & b_0 & 0 \\
         b_0 & c_1 & b_1 & \ddots \\
         0 & b_1 & c_2 & \ddots\\
         & \ddots & \ddots & \ddots & \ddots 
  \end{array}
  \right],
\end{displaymath}
the Favard theorem yields an orthogonality measure $\D\mu$ such that the orthonormal polynomial system $\mathcal{P}=\{p_n\}_{n\in\bb{Z}_+}$ obeys the three-term recurrence relation \R{eq:2.7}, whereby
\begin{equation}
  \label{eq:3.1}
  \varphi_n(x)=\frac{(-\ii)^n}{\sqrt{2\pi}} \int_{-\infty}^\infty p_n(\xi) \sqrt{w(\xi)} \ee^{\ii x\xi}\D\xi,\qquad n\in\BB{Z}_+,
\end{equation}
where $\D\mu=w\D x$. In particular, it determines $\varphi_0$ uniquely. Alternatively, we can use \R{eq:3.1} just to derive $\varphi_0$ and generate the remaining $\varphi_n$s with \R{eq:2.6}. In addition, $\D\mu$ also determines $\MM{b}$ and $\MM{c}$ through the three-term recurrence relation for $\mathcal{P}$. 

The use of the Fourier transform and the mapping associated with the Favard theorem (which could perhaps be called a Favard transform) is only part of the story. There are three sets at play here:
\begin{enumerate}
	\item $\mathcal{S}$ is the set of all possible smooth seed functions $\varphi_0$ (as in Definition \ref{def:seed}) that satisfy $\|\varphi_0\|_{\CC{L}_2(\BB{R})} = 1$.
	\item $\GG{B}$ is the set of all Borel measures on $\BB{R}$, up to equvalence of moments. Here we need to exercise caution: the Favard theorem need not associate a {\em unique\/} Borel measure with a given three-term recurrence relation. A standard example of non-uniqueness is provided by the Stieltjes--Wigert measure \cite[p.~172--5]{chihara78iop}. Therefore the elements of $\mathcal{B}$ are either singleton sets with a single determinate measure, or an equivalence class of measures with equal moments (and hence equal three-term recurrence relations).
	\item $\DD J$ is the set of all irreducible Jacobi matrices, alternatively of the coefficients of a three-term recurrence relation coefficients \R{eq:2.7}. 
\end{enumerate}

\begin{figure}[ht!]
	\centering
    \begin{picture}(210,170)
      \thicklines
      \put (-12,10) {$\GG{B}$}
      \put (10,21) {\vector(1,0){180}}
      \put (190,14) {\vector(-1,0){180}}
      \put (195,10) {${\DD J}$}
      \put (-6,28) {\vector(2,3){90}}
      \put (90,160) {\vector(-2,-3){88}}
       \put (195,23) {\vector(-2,3){90}} 
      \put (110,162) {\vector(2,-3){88}} 
      \put (94,166) {$\mathcal{S}$}
      \put (60,-2) {\small Favard's theorem}
      \put (55,28) {\small Polynomial Lanczos}
      \put (-5,50) {\rotatebox{56}{\small Inverse Fourier transform}}
      \put (29,55) {\rotatebox{56}{\small Fourier transform}}
      \put (143,123) {\rotatebox{304}{\small Differential Lanczos}}
    \end{picture}
    \caption{The commuting triangle of mappings between the set of seeds $\mathcal{S}$, the set of determinate Borel measures $\mathcal{B}$ and the set of non-separable Jacobi matrices $\DD J$.}\label{fig:triangle}
\end{figure}
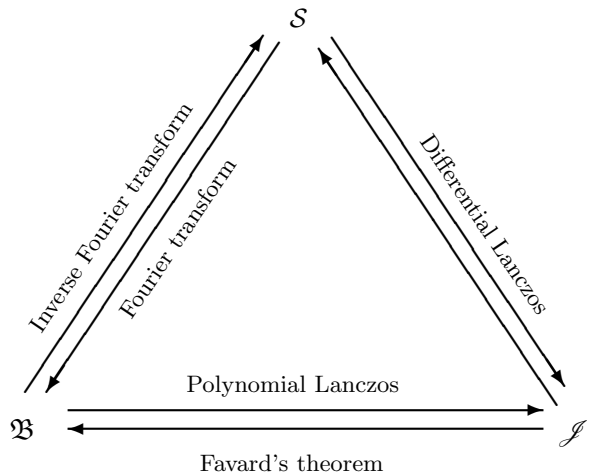

The arrow from $\GG{B}$ to $\DD J$ is classical; it can be realised by the Lanczos iteration applied to the operator $f(x) \mapsto x f(x)$ starting from an initial function $1$ (see Table 8 of \cite{dubbs2015infinite}). Its inverse is the Favard theorem: while its standard proof can be considered as `constructive'  (at a stretch) \cite[p.~21--22]{chihara78iop}, the implied algorithm via a weak limit of discrete measures is impractica. \citeasnoun{colbrook2021computing} present a numerically robust algorithm that constructs  an approximate Borel measure corresponding to a given (infinite) Jacobi matrix.. 

The arrows between $\mathcal S$ and $\GG{B}$ are provided by the Fourier transform and are implicit in the material of Subsection~2.3. This leaves out the arrow from $\mathcal S$ to $\DD J$, which we will discuss below.

\begin{definition}
	Let $\varphi_0$ be a seed function as in Definition \ref{def:seed} and let $\langle \cdot, \cdot \rangle$ be an inner product on $ \bigcup_{n \in \BB{Z}_+}\mathcal{K}_n(\varphi_0)$. We define the \emph{differential Hilbert space}
	$$
	\mathcal{H}(\varphi_0) := \overline{\CC{span}\{\varphi_0^{(n)}\}_{n \in \BB{Z}_+}},
	$$ 
	where the closure is taken with respect to the norm $\|\cdot\|$ induced by the inner product.
\end{definition}

\begin{algorithm}
	\begin{algorithmic}[1]
		\Require A seed function $\varphi_0$ and its associated differential Hilbert space $\mathcal{H}(\varphi_0)$, where the inner product $\langle \cdot, \cdot \rangle$ satisfies the IbP property. $\|\varphi_0\| = 1$
		\Ensure $\Phi = \{ \varphi_n\}_{n=0}^\infty$, $\{b_n\}_{n=0}^\infty$, $\{c_n\}_{n=0}^\infty$.
		\State $b_{-1}\varphi_{-1} \equiv 0$
		\For{$n = 0,1,2,\ldots $}
		\State $c_n = \ii \langle  \varphi_n',  \varphi_n \rangle$
		\State $\psi_{n+1} = \varphi_n' + b_{n-1} \varphi_{n-1} - \ii c_n \varphi_n$
		\State $b_{n} = \| \psi_{n+1} \|$
        \State $\varphi_{n+1} = b_{n}^{-1}\psi_{n+1}$
		\EndFor
	\end{algorithmic}\caption{The Differential Lanczos Algorithm}\label{alg:lanczos}
\end{algorithm}

In deference to the familiar {\em Lanczos algorithm\/} from numerical linear algebra \cite[p.~477]{golub96mc}, we call this the {\em differential Lanczos algorithm.\/} It is merely the Lanczos algorithm applied to the operator $\ii \frac{\CC{d}}{\CC{d}x}$. The IbP property ensures that this operator is Hermitian, or else the Lanczos algorithm is invalid.

\begin{theorem}\label{thm:lanczos_works}
 Algorithm \ref{alg:lanczos} generates the sequences $\Phi=\{\varphi_n\}_{n=0}^
\infty$, $\mathbf{b} = \{b_n\}_{n=0}^\infty$ and $\mathbf{c} = \{ c_n\}_{n=0}^\infty$ such that for all $n = 1,2,\ldots$,
	\begin{enumerate}[(i)]
		\item $\{ \varphi_0, \ldots, \varphi_{n}\}$ is an orthonormal basis for $\mathcal{K}_{n}(\varphi_0)$
		\item $ \varphi_{n-1}'(x) = -b_{n-2} \varphi_{n-2}(x) + \ii c_{n-1} \varphi_{n-1}(x) + b_{n-1} \varphi_{n}(x)$
	\end{enumerate}
	with $b_{-1}\varphi_{-1} \equiv 0$ ensuring the case $n = 1$ is well-defined.
\end{theorem}
\begin{proof}
		Let us prove the result by induction. We start with the base case $n = 1$. We have by line 4,
	\begin{equation}
		\psi_1 = \varphi_0' - \ii c_0 \varphi_0.
	\end{equation} 
	This function is not identically zero because $\varphi_0' \in \mathcal{K}_1(\varphi_0) \setminus \mathcal{K}_0(\varphi_0)$ by Definition \ref{def:seed}, and $\varphi_0 \in \mathcal{K}_0(\varphi_0)$. This implies that $b_0$ is non-zero (line 5), so $\varphi_1$ is well-defined and normalized by line 6 of the algorithm. Furthermore, property $(ii)$ is satisfied by combining lines 4, 5 and 6 of the algorithm. Finally, $\varphi_1$ is orthogonal to $\varphi_0$ because:
	\begin{Eqnarray*}
		\langle \varphi_1, \varphi_0 \rangle &=& b_0^{-1} \langle \varphi_0' - \ii c_0 \varphi_0, \varphi_0 \rangle \\
		 &=& b_0^{-1} \left( \langle \varphi_0', \varphi_0\rangle - \ii c_0 \right) \\
		 &=& 0,
	\end{Eqnarray*}
	by the definition of $c_0$.
	
	Now we assume that properties $(i)$ and $(ii)$ hold for a given $n$, and wish to show it holds when $\varphi_{n+1}$ is generated by the Algorithm \ref{alg:lanczos}. Just as in the base case, $\psi_{n+1}$ is not identically zero, because $\varphi_n' \in \mathcal{K}_{n+1}(\varphi_0) \setminus \mathcal{K}_n(\varphi_0)$ by Definition \ref{def:seed}, and $b_{n-1}\varphi_{n-1} - \ii c_n \varphi_n \in \mathcal{K}_n(\varphi_0)$ by the inductive hypothesis. Therefore, $b_n$ is non-zero and $\varphi_{n+1}$ is well-defined and normalized. As in the base case, property $(ii)$ is satisfied by combining lines 4, 5 and 6 of the algorithm.
	
	To complete the proof, we need to show that $\psi_{n+1}$ is orthogonal to $\mathcal{K}_n(\varphi_0)$, for then $\{\varphi_0, \varphi_1,\ldots,\varphi_{n+1}\}$ forms an orthonormal basis for $\mathcal{K}_{n+1}(\varphi_0)$ by combining this fact with the inductive hypothesis. Now, note that both parts of the inductive hypothesis imply $\langle \varphi_{j}', \varphi_n \rangle = 0$ for $j = 0,1,\ldots,n-2$ and $\langle \varphi_{n-1}', \varphi_n \rangle = b_{n-1}$. Also, the IbP property implies that $\langle f',g \rangle = - \langle f, g'\rangle$ for all $f,g \in \mathcal{H}(\varphi_0)$. Using these facts, we show the following.
	\begin{displaymath}
		\langle \psi_{n+1} , \varphi_{n}\rangle = \langle \varphi_n' + b_{n-1}\varphi_{n-1} - \ii c_n \varphi_n, \varphi_n \rangle = \ii c_n - 0 - \ii c_n = 0.
	\end{displaymath}
	\begin{eqnarray*}
		\langle \psi_{n+1}, \varphi_{n-1} \rangle &=&\langle \varphi_n', \varphi_{n-1} \rangle + b_{n-1} \langle \varphi_{n-1}, \varphi_{n-1}\rangle - \ii c_n \langle \varphi_n,\varphi_{n-1} \rangle \\
		&=& - \langle \varphi_{n}, \varphi_{n-1}' \rangle + b_{n-1} - 0 \\
		&=& -b_{n-1} + b_{n-1} = 0.
	\end{eqnarray*}
	For all $j =0,1,\ldots,n-2$,
	\begin{eqnarray*}
		\langle \psi_{n+1}, \varphi_{j} \rangle &=&\langle \varphi_n', \varphi_{j} \rangle + b_{n-1} \langle \varphi_{n-1}, \varphi_j\rangle - \ii c_n \langle \varphi_n, \varphi_j \rangle \\
		&=&  -\langle  \varphi_n, \varphi_j' \rangle + 0 - 0 \\
		&=& 0.
	\end{eqnarray*}
	This completes the proof.
\end{proof}

Note that this whole procedure was perfectly abstract. At no point did we require that the Hilbert space is a subspace of $\mathrm{L}_2(\BB{R})$, and indeed we will show some examples where this setting is not used. Furthermore, it appears possible to use a degenerate, symmetric bilinear (or Hermitian sesquilinear) product instead of an inner product, as long as $b_n \neq 0$ for all $n$. Indeed, non-degeneracy was only used to prove that $b_n \neq 0$. However, the IbP property was certainly necessary, which is a limitation. We show in Section 4 that the approach can be generalised for degenerate, Hermitian sesquilinear forms that do not satisfy the IbP property.

\subsection{Examples}\label{sec:DLexamples}

We present in this section six different examples in which T-functions are derived using the differential Lanczos algorithm. In five of these examples the triangle in Figure \ref{fig:triangle} is valid, hence we have the alternative of either applying the differential Lanczos algorithm or using a Fourier transform (or series), followed by the Favard theorem (that is, moving round the commutative triangle in the opposite direction). In Example~4, though, this alternative does not exist and the differential Lanczos algorithm is the only currently known means to construct T-functions.

\subsubsection{Example 1: Hermite functions}

Hermite functions
\begin{equation}
  \label{eq:3.3}
  \varphi_n(x)=(-1)^n\frac{\ee^{-x^2/2}}{\sqrt{2^nn!\sqrt{\pi}}} \CC{H}_n(x),\qquad n\in\BB{Z}_+,
\end{equation}
where $\CC{H}_n$ is the standard Hermite polynomial, are the oldest and most well-known example of T-functions. They obey the differential recurrence \R{eq:2.5} with $b_n=\sqrt{(n+1)/2}$ and $c_n\equiv0$. 

Let us work through a couple of steps of Algorithm \ref{alg:lanczos} by hand. We
\begin{Eqnarray*}
	c_0 &=& \ii \langle \varphi_0' \varphi_0\rangle = -\frac{\ii}{\sqrt{\pi}} \int_{-\infty}^\infty x \ee^{-x^2} \D x = 0 \\
	\psi_1 &=& \varphi'_0 - \ii c_0 \varphi_0 = \varphi'_0 = - \pi^{-1/4}  x \ee^{-x^2 / 2} \\
	b_0 &=& \|\psi_1\| = \frac{1}{\sqrt{2}} \\
	\varphi_1 &=& \sqrt{2} \psi_1 = - \sqrt{2} \pi^{-1/4} x \ee^{-x^2 / 2} \\
	c_1 &=&  \ii \langle \varphi_1', \varphi_1\rangle = 0 \text{ (by odd symmetry, as for $c_0$)} \\
	\psi_2 &=& \varphi'_1 + b_0\varphi_0 - \ii c_1 \varphi_1 \\
	            &=& \sqrt{2} \pi^{-1/4}(x^2-1) \ee^{-x^2 / 2} +  \pi^{-1/4}2^{-1/2} \ee^{-x^2 / 2} \\
	           &=& \pi^{-1/4}2^{-1/2} (2x^2 - 1) \ee^{-x^2 / 2} \\
	 b_1 &=& \|\psi_2\| = 1 \\
	 \varphi_2 &=&  \pi^{-1/4}2^{-1/2} (2x^2 - 1) \ee^{-x^2 / 2}.
\end{Eqnarray*}
and so on.

\subsubsection{Example 2: Malmquist--Takenaka functions}

We set 
\begin{displaymath}
  \varphi_0(x)=\sqrt{\frac{2}{\pi}}\frac{1}{1-2\ii x},\qquad x\in\BB{R}. 
\end{displaymath}
Our contention is that underlying T-system is
\begin{displaymath}
  \varphi_n(x)=\sqrt{\frac{2}{\pi}} \ii^n \frac{(1+2\ii x)^n}{(1-2\ii x)^{n+1}}
\end{displaymath}
and $b_n=n+1$, $c_n=2n+1$. This is the {\em Malmquist--Takenaka system\/} \cite{iserles20for} where, to obtain a basis of $\CC{L}_2(\BB{R})$, we need to take $n\in\BB{Z}$. Here we just consider $n\in\BB{Z}_+$, commenting that the remaining values can be obtained running the recursion backwards from $\varphi_{-1}=\overline{\varphi}_0$. 

Assuming that $\varphi_0,\ldots,\varphi_n$ have been already computed, together with $b_m,c_m$ for $m\leq n-1$, we commence by computing $c_n$ from line 3 of Algorithm \ref{alg:lanczos}. We have
\begin{displaymath}
  c_n=\Im \int_{-\infty}^\infty \varphi_n'(x)\overline{\varphi_n(x)}\D x=\Im \frac{4\ii}{\pi} \int_{-\infty}^\infty (2n+1+2\ii x) \frac{\D x}{(1+4x^2)^2}.
\end{displaymath}
But
\begin{displaymath}
  \int_{-\infty}^\infty \frac{\D x}{(1+4x^2)^2}=\frac{\pi}{2},\qquad \int_{-\infty}^\infty \frac{x\D x}{(1+4x^2)^2}=0
\end{displaymath}
and we recover $c_n=2n+1$. 

Now let us evaluate $b_n$. We have
\begin{displaymath}
	\psi_{n+1} = \varphi_n' + b_{n-1} \varphi_{n-1} - \ii c_n \varphi_n.
\end{displaymath}
Rather than find the norm of this directly, we can use the Pythagorean Theorem on the orthgonal functions $\psi_{n+1}$, $\varphi_n$ and $\varphi_{n-1}$ to obtain
\begin{displaymath}
	\|\psi_{n+1}\|^2 = \|\varphi_n'\|^2 - b_{n-1}^2 - c_n^2.
\end{displaymath}
Next, it is a straightforward integral and simple algebra confirms that $\|\varphi_n'\|^2=6n^2+6n+2$. This in turn yields $b_n=n+1$ and the recursive step is done. 

\subsubsection{Example 3: Periodic boundary conditions}

We commence from the periodic function
\begin{equation}
  \label{eq:3.4}
  \varphi_0(x)=\sqrt{\frac{3}{10\pi}} \frac{3}{5-4\cos x},\qquad |x|\leq\pi,
\end{equation}
normalised so that $\|\varphi_0\|_{\CC{L}_2(-\pi,\pi)}=1$, and generate $\varphi_1,\varphi_2,\ldots$ recursively using Algorithm~1. Along the way we also compute the $b_n$s -- of course, $c_n\equiv0$, since $\varphi_0$ is an even function. We obtain\\[4pt]
\begin{tabular}{ll}
  $\displaystyle b_0=\frac{2\sqrt{2}}{3},$ & $\displaystyle \varphi_1(x)=-\frac{9\sqrt{15}}{5\sqrt{\pi}} \frac{\sin x}{(5-4\cos x)^2};$\\
  $\displaystyle b_1=\frac73,$ & $\displaystyle \varphi_2(x)=\frac{3\sqrt{15}}{35\sqrt{\pi}} \frac{122-125\cos x-4\cos^2x}{(5-4\cos x)^3};$\\
  $\displaystyle b_2=\frac{4\sqrt{33}}{7},$ & $\displaystyle \varphi_3(x)=-\frac{18\sqrt{55}}{55\sqrt{\pi}} \frac{(86-125\cos x+32\cos^2x)\sin x}{(5-4\cos x)^4};$\\
  $\displaystyle b_3=\frac{2\sqrt{2556345}}{693}$, & $\displaystyle \varphi_4(x)=$ \footnotesize$\displaystyle \frac{81\sqrt{46439}}{542255\sqrt{\pi}} \frac{34633\!-\!85000\cos x\!+\!57990\cos^2x\!-\!5000\cos^3x\!-\!2560\cos^4x}{(5-4\cos x)^5}$
\end{tabular}
and so on. The  general pattern is
\begin{displaymath}
  \varphi_{2n}(x)=\frac{r_{2n}(\cos x)}{(5-4\cos x)^{2n+1}},\qquad \varphi_{2n+1}(x)=\frac{s_{2n}(\cos x)\sin x}{(5-4\cos x)^{2n+2}},
\end{displaymath}
where $r_m$ and $s_m$ are polynomials of degree $m$ that obey the recurrences
\begin{Eqnarray*}
  b_{2n-1}r_{2n}(y)&=&b_{2n-2}(5-4y)^2 r_{2n-2}(y)-(1-y^2)(5-4y)s_{2n-2}'(y)\\
  &&\hspace*{15pt}\mbox{}-[8n-5\cos x-4(2n-1)\cos^2x] s_{2n-2}(y),\\
  b_{2n}s_{2n}(y)&=&-(5-4y)r_{2n}'(y)+b_{2n-1}(5-4y)^2 s_{2n-2}(y)-4(2n+1)r_{2n}(y).
\end{Eqnarray*}

\begin{figure}[htb]
  \begin{center}
    \includegraphics[width=250pt,height=200pt]{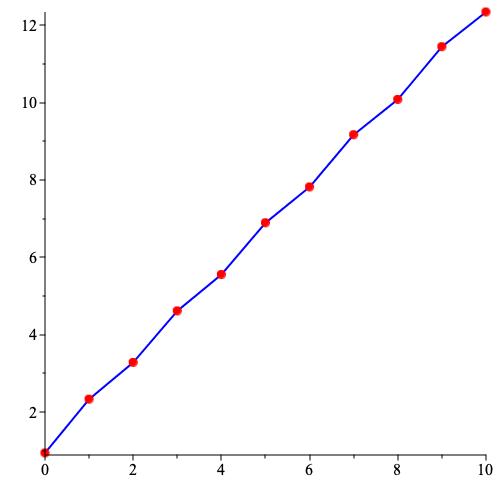}
    \caption{$b_n$ for $n=0,\ldots,10$ and the seed \R{eq:3.4}.}
    \label{fig:3.1}
  \end{center}
\end{figure}

It is fairly straightforward to produce $\{\varphi_n\}_{n=0}^N$ for any reasonable value of $N$. The first coefficients $b_n$ are displayed in Fig.~\ref{fig:3.1} and we note that they increase linearly. Note that we do not know what is the Borel measure associated with the $b_n$s (and $c_n\equiv0$), but this is an advantage of the differential Lanczos algorithm: all we need is a seed and an inner product obeying the IbP property.

\subsubsection{Example 4: Essential singularities}

We return to the setting of Subsection~2.5 and the seed \R{eq:2.17}. This is an example of a situation whereby there is no Fourier transformation taking \R{eq:2.5} to a three-term recurrence relation for orthonormal polynomials, yet the differential Lanczos algorithm can be used to construct $\Phi$ in a recursive manner. In particular, it has been used to derive the $\varphi_n$s in Fig.~\ref{fig:2.3}. 

\begin{figure}[htb]
  \begin{center}
    \includegraphics[width=250pt,height=200pt]{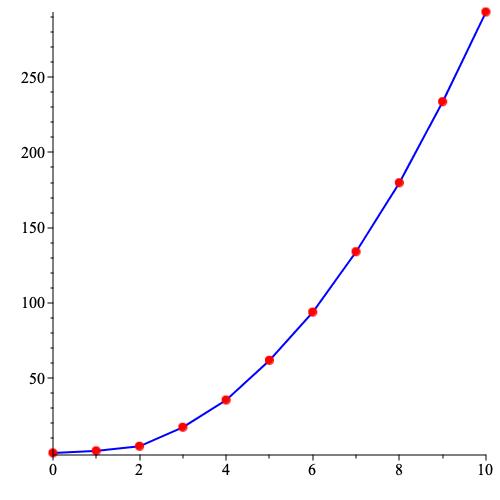}
    \caption{$b_n$ for $n=0,\ldots,10$ and the seed \R{eq:2.17}.}
    \label{fig:3.2}
  \end{center}
\end{figure}

Fig.~\ref{fig:3.2} displays $b_n$, $n=0,\ldots,10$, for the seed \R{eq:2.17} (because $\varphi_0$ is even, we have $c_n\equiv0$). The rate of increase appears to be quadratic. Although the underlying Borel measure is unknown, $\lim b_n=+\infty$ implies, using the Blumenthal theorem \cite{chihara78iop} that it is supported on all of $\BB{R}$.

The differential Lanczos algorithm is considerably more than a computational procedure, whereby each turn of the handle delivers a new $\varphi_{n+1}$, together with $b_n$ and $c_n$. It can also allow us to obtain a great deal of information on the $\varphi_n$s. We have seen it in the previous example.  Insofar as the seed \R{eq:2.5} is concerned, the formal recursion of the differential Lanczos algorithm yields the form
\begin{equation}
  \label{eq:3.13}
  \varphi_n(x)=\frac{r_n(x)}{(1-x^2)^{2n}} \exp\!\left(-\frac{1}{1-x^2}\right)\!,\qquad n\in\BB{Z}_+,
\end{equation}
where each $r_n$ is a polynomial of degree $\rho_n$ and of parity $n$. We have the recurrence
\begin{Eqnarray*}
  r_0&\approx&2.7413264626,\qquad r_1(x)\approx-3.452446691 x,\\
  r_n(x)&=&\frac{1}{b_n} \{ b_{n-1}(1-x^2)^4 r_{n-2}(x)+2[(2n-3)-2(n-1)x^2] xr_{n-1}(x) \\
  &&\hspace*{20pt}\mbox{}+(1-x^2)^2 r_{n-1}'(x)\}.
\end{Eqnarray*}
It follows that
\begin{equation}
  \label{eq:3.14}
  \rho_n\leq\max\{ \rho_{n-2}+8,\rho_{n-1}+3\},\qquad n\geq2,
\end{equation}
with $\rho_0=0$, $\rho_1=1$. Extensive numerical experimentation indicates that $\rho_{2n}=4n$ and $\rho_{2n+1}=4n+1$ -- thus, \R{eq:3.14} holds as an equality for even $n$. When $n$ is odd the highest-order terms cancel and the degree $\rho_n$ drops. 

\subsubsection{Example 5: A Sobolev inner product}

In \cite{iserles23sos} we considered the nonstandard Sobolev inner product,
\begin{displaymath}
  \langle f,g\rangle=\int_{-\infty}^\infty \hat{f}(\xi)\overline{\hat{g}(\xi)}(1+\xi^2)\ee^{-\xi^2}\!\D \xi,
\end{displaymath}
where $\hat{\cdot}$ denotes the continuous Fourier transform on the real line. Therein we derived the T-functions $\Phi$ {\em via\/} the Fourier transform, along the lines sketched in Section~2.3. Here we outline the alternative approach that uses the differential Lanczos algorithm. We commence from the seed
\begin{displaymath}
  \varphi_0(x)=\frac{\sqrt{6}}{3\pi^{1/4}} \ee^{-x^2/2}.
\end{displaymath}
The coefficient in front ensures that $\|\varphi_0\|^2=1$ in the Sobolev norm. Using Algorithm~1 and bearing in mind that $c_n\equiv0$ because the seed is an even function, we generate the pairs $\{\varphi_{n+1},b_n\}$ for $n=1,2,\ldots$\@ The outcome is the same as in \cite{iserles23sos}, where the first few $\varphi_n$s are listed and plotted, except that the point of departure in that paper was (in the notation of Section~2.3) $w(\xi)=(1+\xi^2)\ee^{-\xi^2}$ and $v(\xi)=1+\xi^2$. Using the differential Lanczos algorithm, though, neither the knowledge of $w$ nor of the orthonormal polynomial system with respect to $w(\xi)\D\xi$ are required -- just the seed and the nature of the Sobolev inner product. 

\subsubsection{Example 6: Schr\"odinger evolutions}

The solution of the {\em  linear Schr\"odinger equation in semiclassical coordinates,\/}
\begin{equation}
  \label{eq:3.7}
  \ii\varepsilon \frac{\partial u}{\partial t}=-\varepsilon^2\frac{\partial^2u}{\partial x^2}+V(x) u,\qquad x\in\BB{R},\quad t\geq0,
\end{equation}
where $0<\varepsilon\ll1$ is a small parameter (the square root of the ratio between the mass of an electron and that of a nucleus or several nuclei) and $V$ the interaction potential, can be composed, using splitting methods, from the solutions of two `split' equations,
\begin{equation}
   \label{eq:3.8}
    \ii\varepsilon \frac{\partial u}{\partial t}=-\varepsilon^2\frac{\partial^2u}{\partial x^2},\qquad x\in\BB{R},\quad t\geq0,
\end{equation}
the {\em free Schr\"odinger equation\/}, and the linear ODE
\begin{displaymath}
  \ii\varepsilon \frac{\partial u}{\partial t}=-V(x)u, \qquad x\in\BB{R},\quad t\geq0,
\end{displaymath}
whose solution is $u(x,t)=\exp\left(\ii t\varepsilon^{-1} V(x)\right) u(x,0)$. The solution of the equation \R{eq:3.8} represents more of a challenge, because naive numerical methods on the real line (e.g.\ finite differences or finite elements) require either an infinite number of degrees of freedom or an arbitrary truncation of the infinite space interval. However, as it is easy to verify directly, $u$ is an inverse Fourier transform of  `twisted' Fourier transform of the initial value,
\begin{equation}
  \label{eq:3.9}
  u(x,t)=\frac{1}{\sqrt{2\pi}} \int_{-\infty}^\infty \hat{u}(\xi,0)\ee^{\ii\xi^2\varepsilon t+\ii x\xi}\D\xi,\qquad x\in\BB{R},\quad t\geq0. 
\end{equation}

Spectral methods have a critical advantage once, as in our case, the spatial domain is infinite: in principle, all we need is to expand
\begin{equation}
  \label{eq:3.10}
  u(x,0)=\sum_{n=0}^\infty a_n\varphi_n(x)
\end{equation}
in an orthonormal system $\Phi$ (in practice, of course, we truncate the expansion) and use one of many spectral methods \cite{hesthaven07smt} to evolve the $a_n$s so that $u(x,t)=\sum_{n=0}^\infty \tilde{a}_n(t)\varphi_n(x)$ (of course, $\tilde{a}_n(0)=a_n(0)$).  An alternative approach is described in \cite{iserles23slm}: we expand the initial condition like in \R{eq:3.9} but {\em instead of evolving the $a_n$s we evolve the $\varphi_n$s!\/} In other words,
$u(x,t)=\sum_{n=0}^\infty a_n \chi_n(x,t)$, where, using \R{eq:3.9},
\begin{displaymath}
  \chi_n(x,t)=\frac{1}{\sqrt{2\pi}} \int_{-\infty}^\infty \hat{\varphi}_n(\xi) \ee^{\ii\xi^2\varepsilon t+\ii x\xi}\D\xi,\qquad x\in\BB{R},\quad t\geq0. 
\end{displaymath}

The underlying inner product being $\CC{L}_2(\BB{R})$, we are within the setting of Subsection~2.3. Thus, assuming that $\Phi$ is a T-function system,
\begin{Eqnarray*}
  \varphi_n(x)&=&\frac{\ii^n}{\sqrt{2\pi}} \int_{-\infty}^\infty p_n(\xi)\sqrt{w(\xi)}\ee^{\ii x\xi}\D\xi,\qquad n\in\BB{Z}_+,\\
  \chi_n(x,t)&=&\frac{\ii^n}{\sqrt{2\pi}} \int_{-\infty}^\infty p_n(\xi)\sqrt{\ee^{\ii\xi^2\varepsilon t}w(\xi)}\ee^{\ii x\xi}\D\xi,
\end{Eqnarray*}
where $p_n$ is the $n$th orthonormal polynomial with respect to the weight function $w$. By construction, also $\Xi(t)=\{\chi_n(\,\cdot\,,t)\}_{n\in\bb{Z}_+}$ is a T-system for every $t\geq0$: as a matter of fact, it has been proved in \cite{iserles23slm} that its differentiation matrix is exactly the same as that of $\Phi$. Moreover, the explicit form of the $\chi_n$s has been derived there in the case of Hermite functions \R{eq:3.3}, specifically
\begin{Eqnarray*}
  \chi_n(x,t)&=&\frac{\ii^n}{\sqrt{2^nn!\sqrt{\pi}(1-2\ii\varepsilon t}} \left(\frac{2\ii\varepsilon t+1}{2\ii\varepsilon t-1}\right)^{\!n/2} \exp\!\left(\frac12 \frac{x^2}{2\ii \varepsilon t-1}\right)\!\CC{H}_n\!\left(\frac{x}{\sqrt{1+4\varepsilon t^2}}\right)\\
  &=&\sqrt{\frac{(1+2\ii\varepsilon t)^n}{(1-2\ii\varepsilon t)^{n+1}}} \exp\!\left(-\frac{\ii t\varepsilon x^2}{1+4\varepsilon^2 t^2}\right)\! \varphi_n\!\left(\frac{x}{\sqrt{1+4\varepsilon^2 t^2}}\right)\!,\qquad n\in\BB{Z}_+.\hspace*{30pt}
\end{Eqnarray*}

The triangle in Figure \ref{fig:triangle} applies in the present case. While \R{eq:3.10} has been derived in \cite{iserles23slm} using Fourier transforms (and a great deal of extra algebra), an alternative is to use the differential Lanczos algorithm, commencing from the seed
\begin{displaymath}
  \chi_0(x,t)=\frac{1}{\sqrt{(1-2\ii\varepsilon t)\sqrt{\pi}}} \exp\!\left(\frac12 \frac{x^2}{2\ii \varepsilon t-1}\right)\!.
\end{displaymath}

The differential Lanczos algorithm  is identical to Example~1 (and yields the same sequence $b_0,b_1,\ldots$, while $c_n\equiv0$), except that the algebra is somewhat more complicated. The purpose of the current example is to demonstrate that a complex-valued weight function makes sense within our setting. 

\setcounter{equation}{0}
\setcounter{figure}{0}
\section{H-systems: systems with Hessenberg differentiation matrices}

\subsection{Sesquilinear forms}

Let $\llangle\,\cdot\,,\,\cdot\,\rrangle$ be a symmetric bilinear (or Hermitian sesquilinear) form over some real (or complex) Banach space. To be clear, our convention will be linearity in the first argument and conjugate linearity in the second argument. As we have already commented towards the end of Subsection~3.1, if the form satisfies the IbP property then the differential Lanczos algorithm remains valid, although unless the sesquilinear form is positive definite there is no guarantee that the $b_n$s are always positive. It is possible to obtain $b_n=0$, whereby the process terminates. Let us examine some examples that do not sastisfy the IbP property.

As an example, let us consider again the linear Schr\"odinger equation \R{eq:3.7} given on the real line, except that for simplicity we set $\varepsilon=1$ (the so-called ``atomistic coordinates''). The equation has the Hamiltonian
\begin{displaymath}
  H[u](x,t) = \int_{-\infty}^\infty\left[\left|\frac{\partial u(x,t)}{\partial x}\right|^2+V(x)|u(x,t)|^2\right]\!\D x,
\end{displaymath}
which stays constant for $t\geq0$, i.e.\ is determined solely by the initial condition. In other words, $\llangle u(\,\cdot\,,t),u(\,\cdot\,,t)\rrangle_V=\mbox{const}$, where
\begin{equation}
  \label{eq:4.1}
  \llangle v,w\rrangle_V=\int_{-\infty}^\infty [V(x)v(x)\overline{w(x)}+v'(x)\overline{w'(x)}]\D x.
\end{equation}
Note that \R{eq:4.1} is positive definite if $V(x)\geq0$, $x\in\BB{R}$, but this is not always true: an important example is the Lennard-Jones potential $V(x)=(\sigma/x)^{12}-(\sigma/x)^{6}$ for some $\sigma>0$, modelling the interaction of electronically neutral atoms or molecules, where $\sigma$ is the distance of the interacting particles.

The solution of \R{eq:3.7} conserves both the energy of the system (i.e.~the $\CC{L}_2$ norm of the solution) and Hamiltonian energy. In an ideal situation, we would wish to design a system $\Phi$ which is orthonormal with respect to both sesquilinear forms, so that both energies are preserved. As we have already shown, if $\Phi$ is orthonormal with respect to the $\CC{L}_2(\BB{R})$ inner product, then the $\CC{L}_2$ norm is conserved by a Galerkin scheme, but what if it is orthonormal with respect to the $\llangle\cdot,\cdot\rrangle_V$ sesquilinear form?

We write the solution as 
\begin{displaymath}
	u_N(x,t) = \sum_{n=0}^N a_n(t) \varphi_n(x).
\end{displaymath}
For a Galerkin scheme the right-hand-side is discretised by the matrix
\begin{Eqnarray*}
	L_{j,k} &=& \llangle -\varphi_j'' + V\cdot \varphi_j, \varphi_k, \rrangle_V \\
	&=& \int_{-\infty}^\infty  \left(-\varphi_j''' + V' \varphi_j + V \varphi_j'\right) \overline{\varphi_k'} +V \left(-\varphi_j'' + V\varphi_j \right)  \overline{\varphi_k}  \,\D  x \\
	&=& \int_{-\infty}^\infty  \varphi_j''\overline{\varphi_k''} + 2V\varphi_j' \overline{\varphi_k'}   + V'\left(\varphi_j \overline{\varphi_k'} + \varphi_j' \overline{\varphi_k} \right) + V^2\varphi_j \overline{\varphi_k}  \,  \D  x.
\end{Eqnarray*}
Evidently, this matrix is Hermitian, so the solution to the semidiscretisation, namely
\begin{displaymath}
	a_n(t) = \exp(\ii  t L) a_n(0),
\end{displaymath}
evolves unitarily i.e.~its vector 2-norm is preserved. Therefore,
\begin{Eqnarray*}
	\llangle u_N(\cdot,t),u_N(\cdot,t)\rrangle_V &=& \sum_{n=0}^N |a_n(t)|^2 \qquad \text{(since $\Phi$ is orthonormal w.r.t. $\llangle \cdot,\cdot\rrangle_V$)}\\
	&=& \sum_{n=0}^N |a_n(0)|^2 \qquad \text{ since the evolution is unitary on the coefficients}\\
	&=& \llangle u_N(\cdot,0),u_N(\cdot,0) \rrangle_V.
\end{Eqnarray*}
Hence, the Hamiltonian of the solution to the semi-discretised equation is preserved.

\subsection{Hamiltonian sesquilinear forms}

\subsubsection{Satisfaction of simultaneous invariants}

Recall that the linear Schr\"odinger equation, given as a Cauchy problem, has two important invariants, the $\CC{L}_2$ energy and the Hamiltonian energy \R{eq:4.1}. The following demonstates a fundamental limitation on the associated spectral bases.
\begin{theorem}\label{thm:simul_ortho_norm}
	Let $\Phi$ be a complete orthonormal basis for $\CC{L}_2(\BB{R})$. Then $\Phi$ cannot be orthonormal with respect to $\llangle \cdot, \cdot \rrangle_V$ for any function $V:\BB{R} \to \BB{R}$ that is continuous on some interval $[a,b]$.
	\end{theorem}
	\begin{proof}
Let us assume that $\Phi$ is orthonormal with respect to $\llangle \cdot, \cdot \rrangle_V$. Then $\langle \varphi_n, \varphi_m \rangle = \llangle \varphi_n, \varphi_m \rrangle_V$ for all $m,n$, so $\langle f, g \rangle = \llangle f, g \rrangle_V$ for all $f,g \in \CC{span}(\Phi)$. This implies that $\llangle \cdot, \cdot \rrangle_V$ is an inner product on $\CC{span}(\Phi)$ (because so is $\langle \cdot , \cdot \rangle$) and $\Phi$ is complete in $\CC{L}_2(\BB{R})$ in both inner products. This implies that $\llangle \cdot, \cdot\rrangle_V$ extends continuously to all of $\CC{L}_2(\BB{R})$ and $\llangle f, g\rrangle_V = \langle f,g\rangle$ for all $f,g \in \CC{L}_2(\BB{R})$. Let $f$ be a  smooth function supported in a closed subinterval of $(a,b)$. There are infinitely many such functions $f$. Then for any compactly supported smooth function $g$,
\begin{equation}
	-\langle f'',g\rangle = \langle f', g'\rangle = \llangle f,g\rrangle_V - \langle Vf,g\rangle = \langle (1-V)f,g\rangle.
\end{equation}
Therefore, $f''(x) = (V(x)-1)f(x)$ for all $x \in [a,b]$ and $f(a) = f'(a) = 0$.  Since $V$ is continuous in $[a,b]$, there is only one such $f$ by Picard--Lindel\"of theory, which is a contradiction.
\end{proof}

It is well known in {\em geometric numerical integration\/} that simultaneous conservation of distinct invariants is often impossible except by the exact solution \cite{ge88lph,hairer06gni}. Theorem~\ref{thm:simul_ortho_norm} fits into this pattern, as it implies that it is impossible for a Galerkin spectral method to unconditionally preserve both the $\CC{L}_2(\BB{R})$ norm and the Hamiltonian.

For the sake of greater generality now consider the simultaneous satisfaction of both $H(u)= (u,u)_V$ and of the Sobolev-like norm $\|u\|_v = \sqrt{\langle u,u\rangle_v}$, where
\begin{displaymath}
	\langle f,g\rangle_v =\sum_{k=0}^\infty v_k \int_{-\infty}^\infty f^{(k)}(x) \overline{g^{(k)}(x)}\D x,
\end{displaymath}
with $v_k\geq0$, $k\in\BB{Z}_+$, and $v(x)=\sum_{k=0}^\infty v_k x^{2k}>0$, $x\in\BB{R}$. (Of course, we assume that $v\not\equiv1$.) Note that $v(x)\equiv1$ corresponds to the $\CC{L}_2$ inner product, $v(x)=1+x^2$ to $\CC{H}^1$ etc.\@ and recall the material of Subsection~2.3 on T-systems orthogonal with respect to $\langle\,\cdot\,,\,\cdot\,\rangle_v$. 

\begin{theorem}\label{thm:almost_simul_ortho_norm}
  Suppose that $\Phi$ is a complete orthonormal basis for $\CC{H}_v^1(\BB{R})$ with respect to the inner product $\langle\,\cdot\,,\,\cdot\,\rangle_v$ and orthogonal with respect to the inner product $\llangle \cdot\,,\,\cdot \rrangle_V$,
  \begin{equation}
    \label{eq:4.3}
    \langle \varphi_m,\varphi_n\rangle_v=\delta_{m,n},\quad \llangle \varphi_m,\varphi_n \rrangle_V=\lambda_m\delta_{m,n},\qquad m,n\in\BB{Z}_+.
  \end{equation}
  Then each $\lambda_m$ is an eigenvalue and $\varphi_m$ a corresponding eigenfunction of 
  \begin{equation}
    \label{eq:4.4}
    -y''+Vy=\lambda\sum_{k=0}^\infty (-1)^k v_k y^{(2k)},\qquad x\in\BB{R},
  \end{equation}
  with Cauchy boundary conditions.
\end{theorem}

\begin{proof}
  Let $m,n\in\BB{Z}_+$. Repeatedly integrating by parts,
  \begin{Eqnarray*}
   \llangle \varphi_m,\varphi_n \rrangle_V&=&\int_{-\infty}^\infty [V(x)\varphi_m(x)-\varphi_m''(x)]\overline{\varphi_n(x)}\D x,\\
    \langle \varphi_m,\varphi_n\rangle_v&=&\sum_{k=0}^\infty (-1)^k v_k \int_{-\infty}^\infty \varphi_m^{(2k)}(x) \overline{\varphi_n(x)}\D x.
  \end{Eqnarray*}
  It now follows from \R{eq:4.3} that
  \begin{displaymath}
    \int_{-\infty}^\infty \left[-\varphi_m''(x)+V(x)\varphi_m(x) -\lambda_m \sum_{k=0}^\infty (-1)^k v_k  \varphi_m^{(2k)}(x)\right] \overline{\varphi_n(x)}\D x=0.
  \end{displaymath}
  The theorem follows because $\Phi$ is dense in $\CC{H}_v^1(\BB{R})$.
\end{proof}

Note that, if the Sobolev-like inner product is just the familiar $\CC{L}_2(\BB{R})$ inner product, \R{eq:4.4} reduces to the Sturm--Liouville eigenvalue problem
\begin{equation}
  \label{eq:4.5}
  -y''+Vy=\lambda y,\qquad x\in\BB{R}.
\end{equation}
Being defined on the entire real line, the Sturm--Liouville problem \R{eq:4.5} has a point spectrum if $V$ is real summable in every compact subinterval of $\BB{R}$ \cite{marchenko86slo}. However, this point spectrum is not necessarily complete in $\CC{L}_2(\BB{R})$. A sufficient condition for completeness is that $V$ be locally integrable and $V(x) \to +\infty$ as $|x| \to \infty$ \cite[Thm XI11.67]{reed1978iv}.

\subsubsection{Example: Harmonic potential}

The  example of the {\em harmonic potential\/} $V(x)=x^2$ is indicative of a more general state of affairs which has been the subject of Theorem~\ref{thm:almost_simul_ortho_norm}. Our point of departure are  Hermite functions \R{eq:3.3} and we wish to compute $\rho_{m,n}=\llangle \varphi_m,\varphi_n \rrangle_{x^2}$. Since Hermite functions obey \R{eq:2.5} with $b_n=\sqrt{(n+1)/2}$ and $c_n\equiv0$, we have, by virtue of orthonormality,
\begin{Eqnarray*}
  &&\int_{-\infty}^\infty \varphi_m'(x)\varphi_n'(x)\D x\\
  &=&\frac{\sqrt{mn}}{2} \int_{-\infty}^\infty \varphi_{m-1}(x)\varphi_{n-1}(x)\D x-\frac{\sqrt{m(n+1)}}{2} \int_{-\infty}^\infty \varphi_{m-1}(x)\varphi_{n+1}(x)\D x\\
  &&\hspace*{20pt}\mbox{}-\frac{\sqrt{(m+1)n}}{2} \int_{-\infty}^\infty \varphi_{m+1}(x)\varphi_{n-1}(x)\D x\\
  &&\hspace*{20pt}\mbox{}+\frac{\sqrt{(m+1)(n+1)}}{2} \int_{-\infty}^\infty \varphi_{m+1}(x)\varphi_{n+1}(x)\D x\\
  &=&\left[\frac{\sqrt{mn}}{2}+\frac{\sqrt{(m+1)(n+1)}}{2}\right]\!\delta_{m,n}-\frac{\sqrt{m(n+1)}}{2}\delta_{m-1,n+1}-\frac{\sqrt{(m+1)n}}{2}\delta_{m+1,n-1}.
\end{Eqnarray*}
The standard three-term recurrence relation for orthonormal Hermite polynomials implies that
\begin{displaymath}
  x\varphi_n(x)=\frac{\sqrt{n}}{2}\varphi_{n-1}(x)+\sqrt{\frac{n+1}{2}} \varphi_{n+1}(x),
\end{displaymath}
therefore
\begin{Eqnarray*}
  &&\int_{-\infty}^\infty x^2\varphi_m(x)\varphi_n(x)\D x\\
  &=&\frac{\sqrt{mn}}{2} \int_{-\infty}^\infty \varphi_{m-1}(x)\varphi_{n-1}(x)\D x+\frac{\sqrt{m(n+1)}}{2} \int_{-\infty}^\infty \varphi_{m-1}(x)\varphi_{n+1}(x)\D x\\
  &&\mbox{}+\frac{\sqrt{(m+1)n}}{2}\int_{-\infty}^\infty \varphi_{m+1}(x)\varphi_{n-1}(x)\D x\\
  &&\hspace*{20pt}\mbox{}+\frac{\sqrt{(m+1)(n+1)}}{2}\int_{-\infty}^\infty \varphi_{m+1}(x)\varphi_{n+1}(x)\D x\\
  &=&\left[\frac{\sqrt{mn}}{2}+\frac{\sqrt{(m+1)(n+1)}}{2}\right]\!\delta_{m,n}+\frac{\sqrt{m(n+1)}}{2}\delta_{m-1,n+1}+\frac{\sqrt{(m+1)n}}{2}\delta_{m+1,n-1}.
\end{Eqnarray*}
Consequently,
\begin{equation}
  \label{eq:4.6}
 \llangle \varphi_m,\varphi_n \rrangle_{x^2}=[\sqrt{mn}+\sqrt{(m+1)(n+1)}]\delta_{m,n}=(2n+1)\delta_{m,n},\qquad m,n\in\BB{Z}_+.
\end{equation}
Thus, $\Phi$ is orthogonal with respect to the Hamiltonian sesquilinear form but it is not normalised. This is consistent with Theorem~\ref{thm:simul_ortho_norm}. Note that \R{eq:4.6} is not just a confirmation of \R{eq:4.5} in this setting but also a reflection of the well known fact that the $\varphi_n$s are eigenfunctions of the linear Schr\"odinger equation with the harmonic potential $V(x)=x^2$.

Of course, it is easy to orthonormalise $\Phi$ with respect to the Hamiltonian sesquilinear form, setting $\psi_n(x)=\varphi_n(x)/\sqrt{2n+1}$. Alas, $\{\psi_n\}_{n\in\bb{Z}_+}$ is no longer orthonormal with respect to the $\CC{L}_2(\BB{R})$ norm and, while its differentiation matrix is tridiagonal,
\begin{displaymath}
  \psi_n'(x)=\sqrt{\frac{n(2n-1)}{2(2n+1)}}\psi_{n-1}(x)-\sqrt{\frac{(n+1)(2n+3)}{2(2n+1)}}\psi_{n+1}(x),\qquad n\in\BB{Z}_+,
\end{displaymath}
it is no longer skew symmetric. 

Spectra of Schr\"odinger operators are known explicitly in few settings, cf.\ for example \cite{ishkhanyan18spo}. For example, the eigenfunctions corresponding to $V(x)=x^4$ can be expressed in terms of {\em Heun functions.\/} In our case, however, the knowledge of the spectrum is just one requirement, the other is that the eigenfunctions form (subject to proper renormalisation) a T-system. Even if, like in the case of the harmonic potential, both conditions are satisfied, Theorem~\ref{thm:simul_ortho_norm} demonstrates that $\Phi$ cannot be simultaneously orthonormal with respect to the two sesquilinear forms. Moreover, there is no reason why the differentiation matrix with respect to eigenfunctions of a differential operator should be tridiagonal, hence T-system conditions in general are not satisfied.

\subsection{Differential Krylov subspaces revisited}

Let $\llangle \cdot, \cdot\rrangle$ be a sesquilinear form. We consider its differentiation matrix $\DDD=(d_{m,n})_{m,n\in\bb{Z}_+}$. In general we cannot expect $\DDD$ to be a tridiagonal matrix and $\llangle \cdot, \cdot\rrangle$ does not have the IbP property unless $V$ is a constant function. Hence the assumptions of the differential Lanczos algorithm do not necesssarily hold. Instead, we should consider an analogue of Arnoldi's algorithm, in which an orthonormal basis of the space is generated so that the differential operator (in place of a square matrix in the classical case) is represented by an upper Hessenberg matrix.

\begin{algorithm}
	\begin{algorithmic}[1]
		\Require  A seed function $\varphi_0$ and a Hermitian sequilinear form $\llangle \cdot,\cdot \rrangle$ on $\mathrm{span}(\{\varphi_0^{(n)}\}_{n\in \bb{Z}_+})$ such that $\llangle \varphi_0, \varphi_0\rrangle = 1$.
		\Ensure $\Phi = \{ \varphi_n\}_{n=0}^\infty$, differentiation matrix $\DDD=(D_{m,n})_{m,n\in\bb{Z}_+}$.
		\For{$n = 0,1,2,\ldots $}
		\State $\psi_{n+1} = \varphi_n'$
		\For{$k = 0,\ldots,n$}
		 \State $D_{k,n} = \llangle \psi_{n+1}, \varphi_k \rrangle$
		 \State $\psi_{n+1} = \psi_{n+1} - D_{k,n} \varphi_k$
		 \EndFor
		\State $D_{n+1,n} = \left| \llangle \psi_{n+1},\psi_{n+1} \rrangle\right|^{1/2}$
		\State $\varphi_{n+1} = D_{n+1,n}^{-1}\psi_{n+1}$
		\EndFor
	\end{algorithmic}\caption{The Differential Arnoldi Algorithm}\label{alg:arnoldi}
\end{algorithm}

\begin{theorem}\label{thm:arnoldi_works}
If $D_{n+1,n} > 0$ for all $n$, then Algorithm \ref{alg:arnoldi} generates $\Phi=\{\varphi_n\}_{n=0}^
	\infty$, and upper Hessenberg matrix $\DDD$ (with elements $D_{k,j}$) such that for all $n = 1,2,\ldots$,
	\begin{enumerate}[(i)]
		\item $\{ \varphi_0, \ldots, \varphi_{n}\}$ is an orthonormal basis for $\mathcal{K}_{n}(\varphi_0)$ with respect to $\llangle\cdot,\cdot \rrangle$
		\item $ \varphi_{n-1}'(x) =\sum_{k=0}^n D_{k, n-1}\varphi_{k}(x)$.
	\end{enumerate}
\end{theorem}
\begin{proof}
	Let us prove the result by induction. We start with the base case $n = 1$. We have by lines 2 to 6,
	\begin{equation}
		\psi_1 = \varphi_0' - D_{0,0}\varphi_0.
	\end{equation} 
	This function is not identically zero because $\varphi_0' \in \mathcal{K}_1(\varphi_0) \setminus \mathcal{K}_0(\varphi_0)$ by Definition \ref{def:seed}, and $\varphi_0 \in \mathcal{K}_0(\varphi_0)$. We have assumed that $D_{1,0}$ is non-zero, so $\varphi_1$ is well-defined and normalized by line 8 of the algorithm. Furthermore, property $(ii)$ is satisfied by combining lines 2 to 8 of the algorithm. Finally, $\varphi_1$ is orthogonal to $\varphi_0$ because:
	\begin{Eqnarray*}
		\llangle \varphi_1, \varphi_0 \rrangle &=& D_{1,0}^{-1} \llangle \varphi_0' -D_{0,0} \varphi_0, \varphi_0 \rrangle \\
		&=& D_{0,0}^{-1} \left( \llangle \varphi_0', \varphi_0\rrangle - D_{0,0} \right),
	\end{Eqnarray*}
which is zero by line 4 of the algorithm.
	
	Now we assume that properties $(i)$ and $(ii)$ hold for a given $n$, and we wish to show it holds when $\varphi_{n+1}$ is generated by Algorithm \ref{alg:arnoldi}. Just as in the base case, we have that $\psi_{n+1}$ is not identically zero, because $\varphi_n' \in \mathcal{K}_{n+1}(\varphi_0) \setminus \mathcal{K}_n(\varphi_0)$ by Definition \ref{def:seed}, and lines 4 and 5 only subtract off elements of $\mathcal{K}_n(\varphi_0)$, by the inductive hypothesis and line 8 of the algorithm. We have assumed that $D_{n+1,n}$ is non-zero, so $\varphi_{n+1}$ is well-defined and normalized. Lines 2 to 8 imply that $D_{n,n-1} \varphi_n = \varphi_{n-1}' - \sum_{k=0}^{n-1} D_{k,n-1} \varphi_k$, so property $(ii)$ is satisfied.
	
	To complete the proof, we need to show that $\psi_{n+1}$ is orthogonal to $\mathcal{K}_n(\varphi_0)$, for then $\{\varphi_0, \varphi_1,\ldots,\varphi_{n+1}\}$ forms an orthonormal basis for $\mathcal{K}_{n+1}(\varphi_0)$ by combining this fact with the inductive hypothesis. We first need to show that $D_{k,n} = \llangle \varphi_n',\varphi_k\rrangle$ for any $0 \leq k \leq n$. This follows from line 4 of the algorithm combined with the orthonormality of $\{\varphi_0,\ldots,\varphi_n\}$. Now, for $j = 0,1,\ldots,n$,
	\begin{Eqnarray*}
		\llangle \psi_{n+1}, \varphi_{j} \rrangle &=& \llangle \varphi_n' - \sum_{k=0}^n D_{k,n} \varphi_k, \varphi_{j}\rrangle \\
		&=&  \llangle \varphi_n', \varphi_j  \rrangle - \sum_{k=0}^n D_{k,n} \llangle \varphi_k,\varphi_j \rrangle  \\
		&=& \llangle \varphi_n', \varphi_j  \rrangle - D_{j,n},
	\end{Eqnarray*}
	by the orthonormality of $\{\varphi_0,\ldots,\varphi_n\}$.  The fact that $D_{j,n} = \llangle \varphi_n',\varphi_j\rrangle$ completes the proof.
\end{proof}

\subsection{Example: an H-system for a quartic potential}

Let us consider the case $V(x) = x^4$ and seed function proportional to the Gaussian $\phi(x) = \ee^{-x^2 / 2}$. We start by normalising $\phi$ with respect to the inner product $\llangle \cdot, \cdot \rrangle_V$:
\begin{equation}
   \|\phi\|_V^2 = \int_{-\infty}^\infty [x^4 \phi(x)^2 + \phi'(x)^2]  \D x = \int_{-\infty}^\infty (x^4 +x^2) \ee^{-x^2} \, \mathrm{d} x = \frac{5}{4}\pi^{1/2},
\end{equation}
so
\begin{equation}
	\varphi_0(x) = \alpha_0 \ee^{-x^2/2} = \frac{2}{5^{1/2}\pi^{1/4}} \ee^{-x^2/2}.
\end{equation}
Then, $\psi_1(x) = \varphi_0'(x) = - \alpha_0 x \ee^{-x^2/2}$. Since $\phi_0$ is even and $\psi_1$ is odd, we have $D_{0,0} = \llangle \psi_1,\varphi_0\rrangle_V = 0$. Now we normalise $\psi_1$ to obtain $\varphi_1$ (we spare you the details):
\begin{equation}
 \varphi_1(x) = \alpha_1 x \ee^{-x^2/2} = -2\frac{2^{1/2}}{21^{1/2} \pi^{1/4}} x \ee^{-x^2/2},
\end{equation}
which came from $D_{1,0} = \|\psi_1\|_V = 21^{1/2}/10^{1/2}$. The next basis function is where we see the non-skew-Hermitian nature of the differentation matrix. We have $\psi_2(x) = \varphi_1'(x) = \alpha_1 (1-x^2) \ee^{-x^2/2}$. We can compute $D_{1,1} = \llangle \psi_2,\varphi_1 \rrangle_V = 0$ by parity, and 
\begin{eqnarray*}
D_{0,1} &=& \llangle \psi_2, \varphi_0\rrangle_V \\
              &=& \int_{-\infty}^\infty \alpha_1 \alpha_0 x^4(1-x^2) \ee^{-x^2} + \left( \alpha_1(x^3 - 3x) \right)\left(-\alpha_0 x \ee^{-x^2/2}\right)  \D x \\
              &=& \frac{3}{210^{1/2}}.
\end{eqnarray*}
Indeed, we see that $D_{0,1} \neq D_{1,0}$. The subsequent basis functions are odd and even polynomials multiplied by $\ee^{-x^2 / 2}$. We plot the first six basis functions (generated using a symbolic implementation of Algorithm \ref{alg:arnoldi} in Julia) in Figure \ref{fig:quarticH}.
\begin{figure}[htb]
	\begin{center}
\includegraphics[width=.7\textwidth]{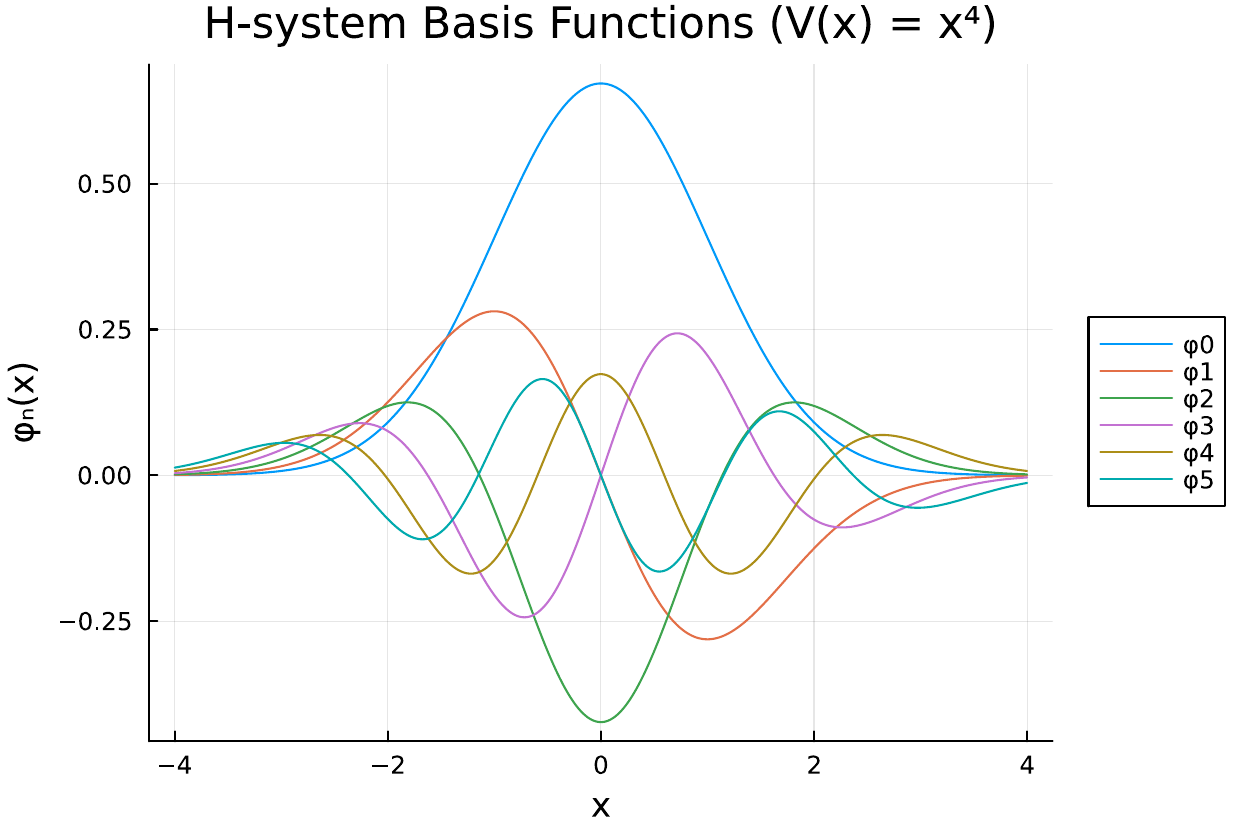}
		\caption{The H-system $\varphi_0,\ldots,\varphi_5$ corresponding to $V(x) = x^4$ and seed function a suitable multiple of $\ee^{-x^2/2}$, generated by the differential Arnoldi algorithm (Algorithm \ref{alg:arnoldi}).}
		\label{fig:quarticH}
	\end{center}
\end{figure}

\subsection{When are H-systems almost-T-systems?}

Of course, $\DDD$ is in general neither skew-symmetric nor tridiagonal. However, an intriguing computational observation is that it misses on both properties by a very small margin. Fig.~\ref{fig:4.1} displays a bar chart of size of the elements of the matrix $|\DDD|=(|D_{m,n}|)_{m,n=0,\ldots,16}$ for four different functions $V$ (that define a sesquilinear form $\llangle \cdot, \cdot \rrangle_V$), the seed function a multiple of $\ee^{-x^2 / 2}$. The most prominent feature are the two leading off-diagonals. (The diagonal is of course zero.) While behind the super-diagonal all elements are zero (the matrix is lower Hessenberg), the terms under the sub-diagonal are very small. In other words, while $|D_{n,n+1}|$ and $|D_{n+1,n}|$ are relatively large (and, although this is not apparent from Fig.~\ref{fig:4.1}, $|D_{n,n+1}+D_{n+1,n}|$ is small), the terms $|D_{m,n}|$ are small for $n\leq m-2$, hardly noticeable to a naked eye. The practical implication of this observation (which, as things stand, is just an observation, with neither proof or even a decent hand-waiving explanation) is that, although the system $\Phi$, orthogonal with respect to the bilinear product \R{eq:4.1}, is neither orthogonal in the $\CC{L}_2(\BB{R})$ inner product nor a T-system, it fails on both counts only slightly.

\begin{figure}[htb]
  \begin{center}
    \hspace*{-20pt}\includegraphics[width=220pt,height=200pt]{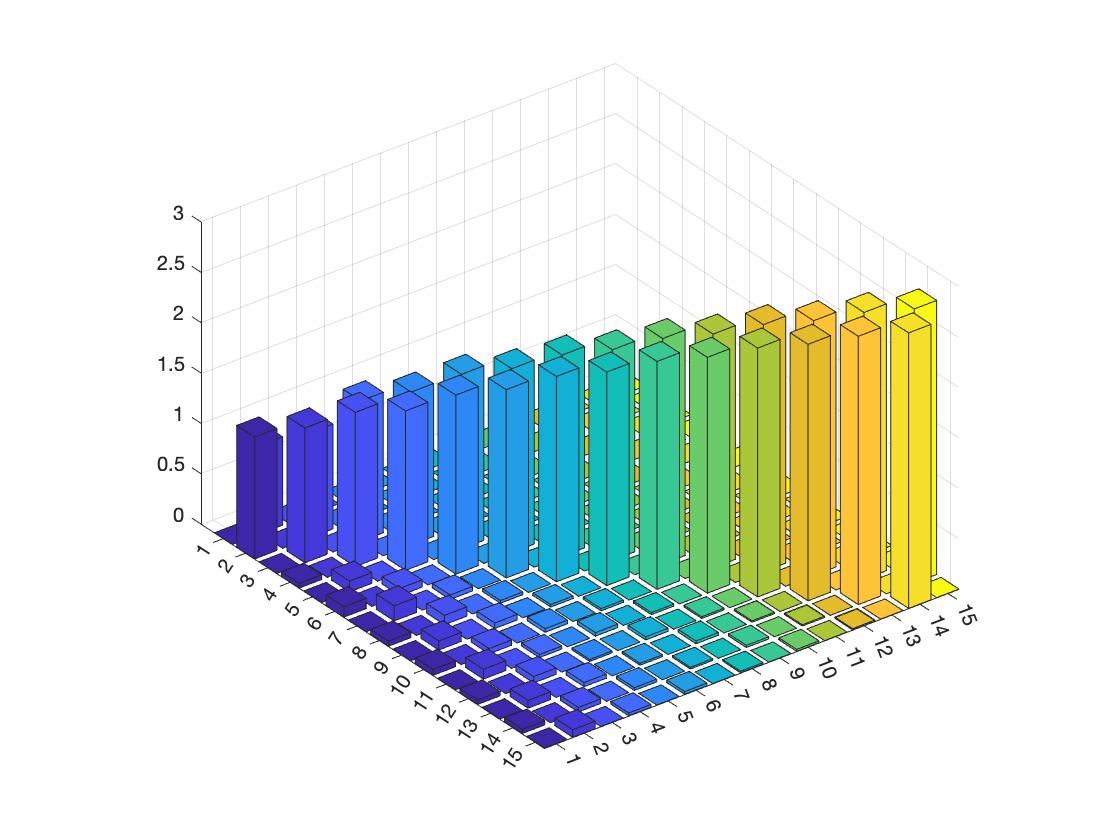}\hspace*{-30pt}\includegraphics[width=220pt,height=200pt]{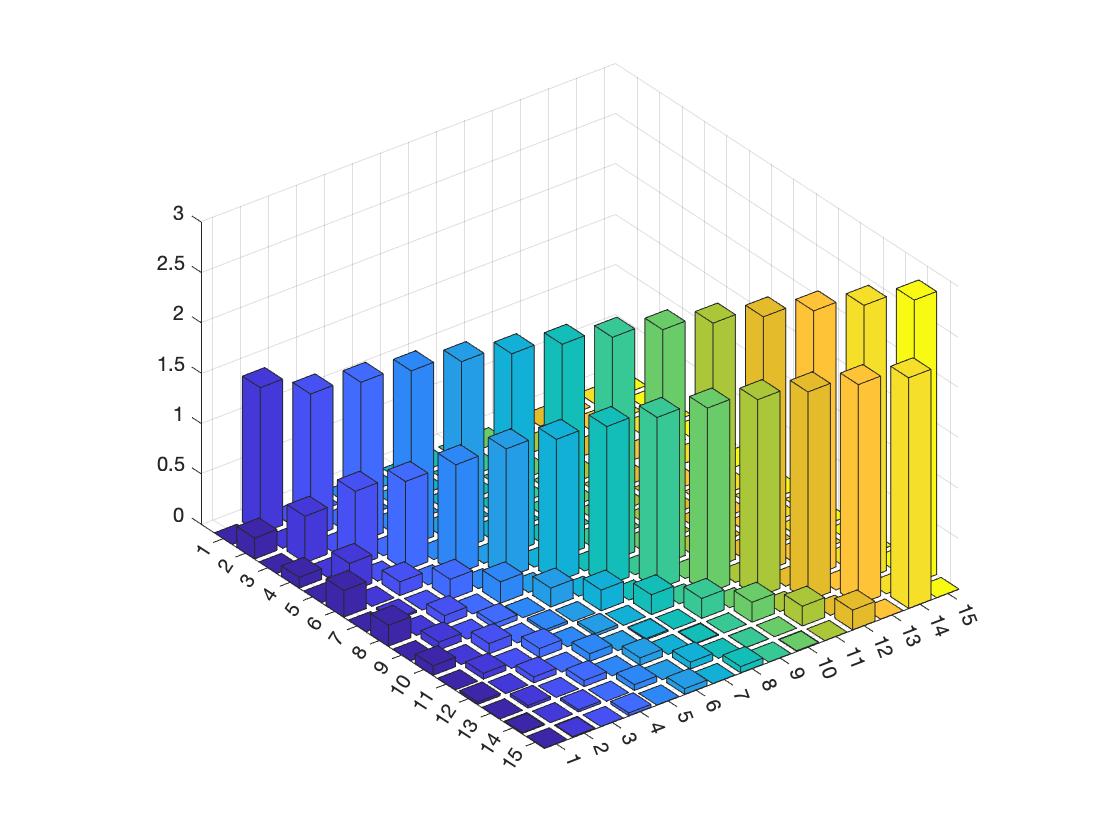}
   
    \vspace*{15pt}
    \hspace*{-20pt}\includegraphics[width=220pt,height=200pt]{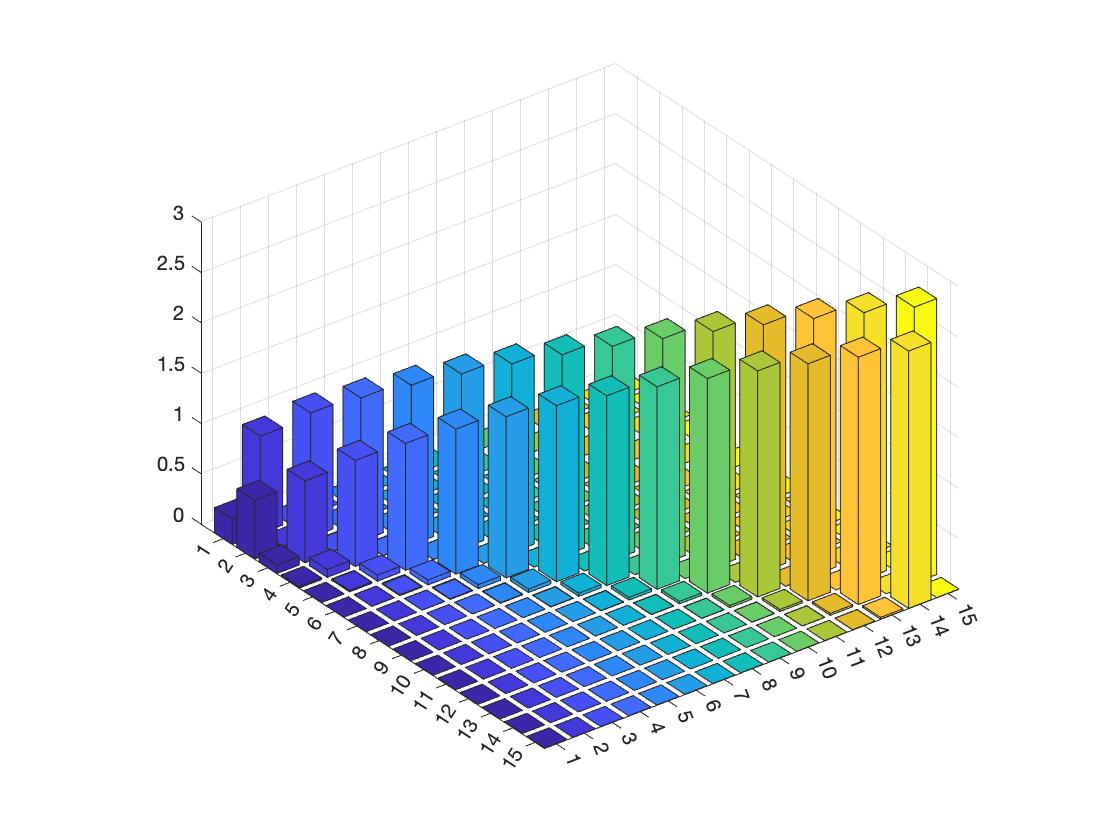}\hspace*{-30pt}\includegraphics[width=220pt,height=200pt]{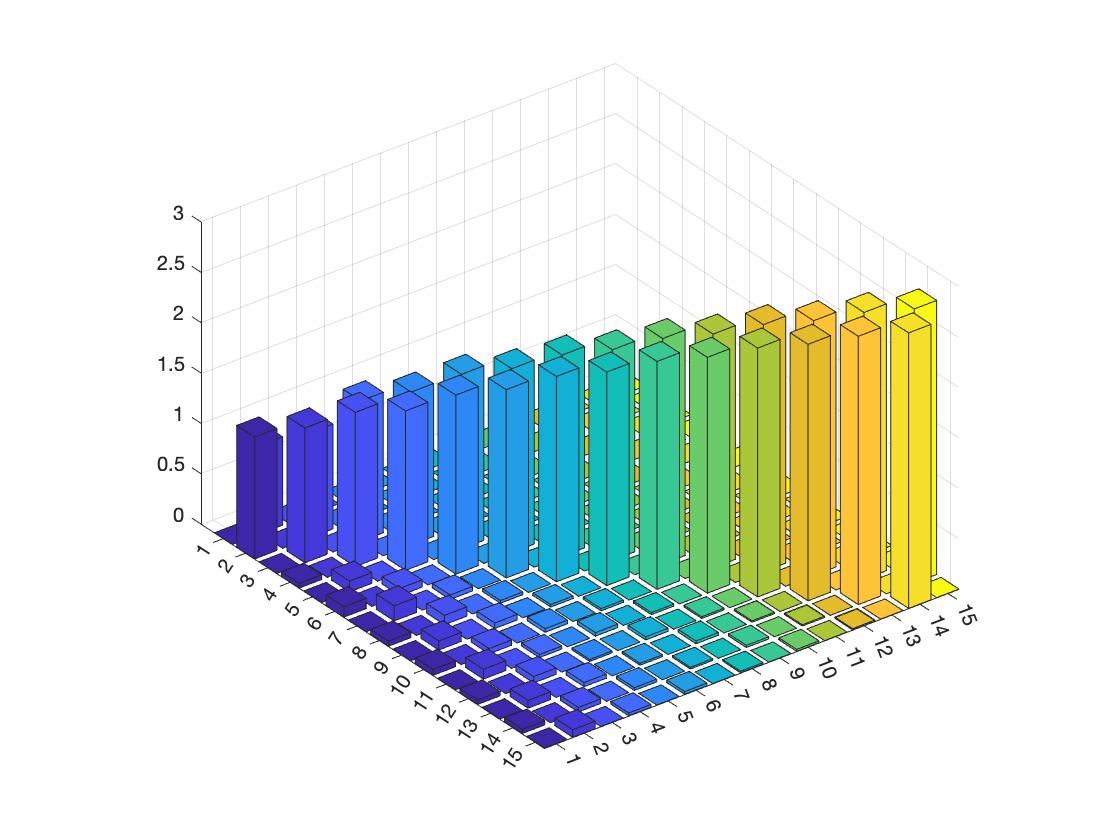}
    \begin{picture}(0,0)
      \put (-115,425) {$V(x)=\ee^{-x^2}$}
      \put (80,425) {$V(x)=x^4$}
      \put (-125,210) {$V(x)=1+x+x^2$}
      \put (70,210) {$V(x)=x^4-2x$}
    \end{picture}
    \caption{A two dimensional bar chart of $|\DDD|$ for four different potentials $V$, starting from a Gaussian seed function.}
    \label{fig:4.1}
  \end{center}
\end{figure}
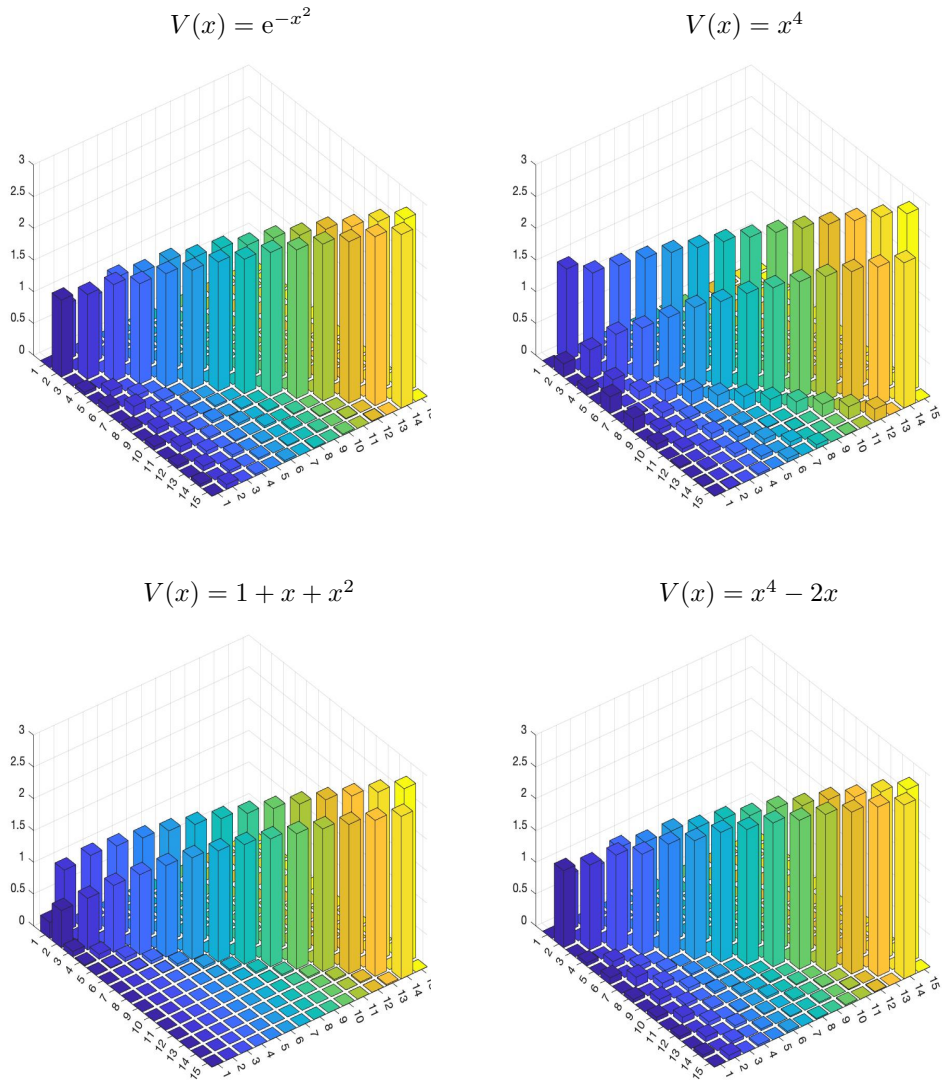

A theoretical conformation of the observation following from Fig.~\ref{fig:4.1} is unavailable, neither is an understanding of its practical implications. Both are a matter for future research.

\section{Conclusion}

We have laid down a concrete theory of orthonormal functions that satisfy a tridiagonal differentiation matrix, which we call T-systems. This begins with characterisations in terms of the Fourier Transform, that the present authors developed over several papers \cite{iserles19oss} \cite{iserles20for}, \cite{iserles21fco}, \cite{iserles22awp}, \cite{iserles23sos}, and then leads to some new observations. 

First, T-systems can be generated using a version of the Lanzcos algorithm, which finds an orthogonal basis tridiagonalising a symmetric matrix, to the differential operator. There are two consequences to this: new T-systems can be generated algorithmically from a smooth seed function without the need for an appropriate differentiation matrixt that would make the system orthogonal and without the Fourier transform; also, the analysis of the approximation properties T-systems may now be tackled using the theory of Krylov subspace approximation! 

Second, we began to develop a methodology around spectral methods that preserve a Hamiltonian form. We showed in section 4.1 that a sufficient condition for this property is to ensure that the basis is orthonormal with respect to the associated sesquilinear form. It turns out that we can indeed generate bases with this kind of orthonormality, but we must use a differential analogue of the Arnoldi Algorithm instead of the Lanczos algorithm (because the differential operator is not necessarily skew-selfadjoint with respect to the Hamiltonian form), and the resulting differentiation matrix is upper-Hessenberg instead of skew-symmetric tridiagonal. We will call these systems H-systems.

We did not investigate H-systems in detail, but there appear to be some interesting properties worthy of investigation. In particular, we found that the differentiation matrix is very close to skew-symmetric tridiagonal, the reasons for which we leave as an open problem.

\bibliographystyle{agsm}
\bibliography{refs}

\end{document}